\documentclass[a4,usenames]{amsart}

\usepackage[utf8]{inputenc} 
\usepackage[T1]{fontenc}
\usepackage[english]{babel}
\usepackage{mathtools}
\usepackage{amssymb}
\usepackage{amsthm}
\usepackage{mathrsfs}
\usepackage{scalerel}

\usepackage{tikz,graphicx}

\usetikzlibrary{automata,positioning}
\usetikzlibrary{calc}

\usepackage{enumerate,url,float,lscape}
\usepackage{hyperref}

\usepackage[normalem]{ulem}

\usepackage{aliascnt}

\theoremstyle{plain}
\newtheorem{theo}{Theorem}[section]

\newaliascnt{cor}{theo}
\newaliascnt{prop}{theo}
\newaliascnt{lemma}{theo}

\newtheorem{lemma}[lemma]{Lemma}
\newtheorem{prop}[prop]{Proposition}
\newtheorem{cor}[cor]{Corollary}

\aliascntresetthe{cor}
\aliascntresetthe{prop}
\aliascntresetthe{lemma}

\theoremstyle{definition} 

\newaliascnt{defi}{theo}
\newaliascnt{assum}{theo}
\newaliascnt{assums}{theo}
\newaliascnt{prob}{theo}

\aliascntresetthe{defi}
\aliascntresetthe{assum}
\aliascntresetthe{assums}
\aliascntresetthe{prob}

\theoremstyle{remark}

\newaliascnt{rems}{theo}
\newaliascnt{rem}{theo}
\newaliascnt{exa}{theo}
\newaliascnt{exs}{theo}

\newtheorem{rem}[rem]{Remark}
\newtheorem{exa}[exa]{Example}

\aliascntresetthe{rems}
\aliascntresetthe{rem}
\aliascntresetthe{exa}
\aliascntresetthe{exs}

\numberwithin{equation}{section}
\numberwithin{lemma}{section}

 \def\mG{\mathsf{G}}

 \def\mV{\mathsf{V}}

 \def\mE{\mathsf{E}}

 \def\mv{\mathsf{v}}

 \def\mw{\mathsf{w}}

\newcommand{\R}{\mathbb{R}}

\DeclareMathOperator{\sign}{sign}
\DeclareMathOperator{\real}{Re}

\newcommand{\C}{\mathbb{C}}
\newcommand{\N}{\mathbb{N}}

\newcommand{\ud}{\,\mathrm{d}}
\newcommand{\e}{\mathrm{e}}

\DeclareMathOperator{\sgn}{sgn}

\DeclareMathOperator{\dist}{dist}

\title[Pointwise eigenvector estimates by landscape functions]
{Pointwise eigenvector estimates by landscape functions:\\ some variations on the Filoche--Mayboroda--van den Berg bound}

\author[D.~Mugnolo]{Delio Mugnolo}

\address{Delio Mugnolo, Lehrgebiet Analysis, Fakult\"at Mathematik und Informatik, Fern\-Universit\"at in Hagen, D-58084 Hagen, Germany}
\email{delio.mugnolo@fernuni-hagen.de}



\thanks{
We are grateful to Bobo Hua (Fudan University, Shanghai) for the computation of the ground state of the dicrete $p$-Laplacian in Section~\ref{exa:main-nonlinear} and to José M.\ Mazón (Valencia) for references concerning general comparison principle for $p$-Laplacians. We warmly thank Jochen Glück (Wuppertal) and Stefan Steinerberger (University of Washington, Seattle) for interesting discussions. Finally, we wish to thank the anonymous referee for pointing us at the bounds on the principal eigenvalue in~\cite{DonVar75,DonVar76}.\\
The author was partially supported by the Deutsche Forschungsgemeinschaft (Grant 397230547).
}

\begin{document}
\begin{abstract}
Landscape functions are a popular tool used to provide upper bounds for eigenvectors of Schrödinger operators on domains. We review some known results obtained in the last ten years, unify several approaches used to achieve such bounds, and extend their scope to a large class of linear and nonlinear operators. 
We also use landscape functions to derive lower estimates on the principal eigenvalue -- much in the spirit of earlier results by Donsker--Varadhan and Ba\~nuelos--Carrol -- as well as upper bounds on heat kernels.
Our methods solely rely on order properties of operators: we devote special attention to the case where the relevant operators enjoy various forms of elliptic or parabolic maximum principles. 
Additionally, we illustrate our findings
with several examples, including $p$-Laplacians on domains and graphs as well as Schrödinger operators with magnetic and electric potential,
also by means of elementary numerical experiments.

\end{abstract}

\maketitle

\section{Introduction}

Several authors have pursued in the last ten years the task of deriving pointwise bounds of the form
\begin{equation}\label{eq:basic-fm}
\frac{|\varphi(x)|}{\|\varphi\|_\infty}\le |\lambda| v(x),\qquad x\in \Omega,
\end{equation}
on \underline{all} eigenpairs $(\lambda,\varphi)$
of suitable differential operators $A$ in terms of \underline{one} auxiliary function $v:\Omega \to (0,\infty)$: such $v$ satisfying \eqref{eq:basic-fm} or, more generally,
\[
\frac{|\varphi(x)|}{\|\varphi\|_\infty} \le c v(x),\qquad x\in \Omega,
\]
 for some $c=c(\lambda)>0$, have been often called \textit{landscape functions} since~\cite{FilMay12}.
In the case of Schrödinger operators, $A:=-\Delta+V$ with $V:\Omega\to (0,\infty)$ and with Dirichlet conditions on the boundary of a bounded 
$\Omega\subset\R^d$, Filoche and Mayboroda have proved in \cite{FilMay12}  that a possible landscape function is given by the unique solution of
\begin{equation}\label{eq:tors}
Av(x)=1,\qquad x\in \Omega;
\end{equation}
the same estimate has been independently observed in~\cite{Ber12} for $V\equiv 0$.
Observe that the Schrödinger operator $A=-\Delta+V$ satisfies a weak maximum principle, that is, $A^{-1}$ is a positive operator: hence, the solution $v$ of~\eqref{eq:tors} is indeed a positive function. Such $v$ is sometimes called the \textit{torsion function} of $\Omega$ (with respect to $\Delta+V$)
and has since~\cite{Pol48} been the subject of many investigations.

In the case of such Schrödinger operators, 
a common interpretation of \eqref{eq:basic-fm} is that any eigenfunction $\varphi$ for the eigenvalue $\lambda$ can only have peaks (``localize'') in the set $\{x\in \Omega:\frac{1}{v(x)}\le \lambda\}$. Because similar localization properties hold for the potential wells $\{x\in \Omega:V(x)\le \lambda\}$ (see, e.g.,~\cite[Theorems~3.4 and 3.10]{HisSig96}), it is tempting to interpret $\frac{1}{v}$ as a proxy for the potential $V$. The fact that $\frac{1}{v}$ is typically smoother than $V$
 and even more effective at predicting localization of eigenfunctions justifies the attention devoted to \eqref{eq:basic-fm} 
in the last years.

\medskip
To fix the ideas at the core of the theory of landscape functions, let us present its fundamental theorem in a general form that holds for operators $A$ on a Lebesgue space $L^2(X)$ that \textit{have dominated inverse}, i.e., such that
\[
|A^{-1}f(x)|\le A^{(-1)}_+|f(x)|\qquad \hbox{for all }f\in L^2(X)\hbox{ and a.e. }x\in X,
\]
for some bounded linear (and necessarily positive) operator $A^{(-1)}_+$ on $L^2(X)$.

\begin{theo}\label{thm:filmay-very-abs}
Let $X$ be a finite measure space, and 
let $A$ be an invertible closed linear operator on $L^2(X)$ with dominated inverse. If $(\lambda,\varphi)$ is an eigenpair of $A$ such that $\varphi\in L^\infty(X)$, then
\begin{equation}\label{eq:basic-fm-superabstr}
|\varphi(x)|\le |\lambda| \|\varphi\|_\infty A^{(-1)}_+ \mathbf{1}(x)\qquad \hbox{for a.e.\ }x\in X.
\end{equation}
\end{theo}
By \eqref{eq:basic-fm-superabstr}, any eigenfunction $\varphi$ of $A$ for the eigenvalue $\lambda$ can only localize in the set $\{x\in X:\frac{1}{A^{(-1)}_+ \mathbf{1}(x)}\le |\lambda|\}$.
\medskip
Theorem~\ref{thm:filmay-very-abs} was proved by van den Berg in~\cite[Theorem~5]{Ber12}
 (for the free Laplacian)
 and by Filoche and Mayboroda in~\cite[Proposition~0.1]{FilMay12b} 
 (for Schrödinger operators with positive potential) 
using  -- beside the maximum principle -- self-adjointness, resolvent compactness, and the existence of a positive Green function (\cite{FilMay12}) or a positive heat kernel (\cite{Ber12}); but their proofs do not actually depend on whether $V\equiv 0$ or not.
Indeed, it was observed by Steinerberger in \cite[page~2903]{Ste17} (and, in a slightly different context, already in the proof of~\cite[Lemma~3]{LyrMayFil15}) that these assumptions can be removed and the proof boils down to a one-liner that can be performed under the sole assumption that the relevant operator has dominated inverse.

\begin{proof}[Proof of \autoref{thm:filmay-very-abs}] Let $\rho:=\mathbf{1}$ 
to see
\begin{equation}\label{eq:musterfilmay}
|\varphi|= |A^{-1}(\lambda\varphi)|\le  A^{(-1)}_+|\lambda\varphi|\le A^{(-1)}_+(|\lambda|\|\varphi\|_\infty \rho)=|\lambda|\|\varphi\|_\infty A^{(-1)}_+ \rho.\qedhere
\end{equation}
\end{proof}

This works upon choosing $A^{(-1)}_+:=A^{-1}$ whenever $A$ has positive inverse (the case treated in~\cite{Ber12,FilMay12,Ste17}) and in fact even for general dominating operators $A^{(-1)}_+$: the latter case has been discussed in the settings of operators whose inverse has a (not necessarily positive!) absolutely integrable kernel in~\cite{FilMay12b}, and of matrices in~\cite{LemPacOvd20}; in the latter case, lower bounds for eigenfunctions in their localization regions are discussed in~\cite{LuSte18}.
 
 \medskip
Being interested in \textit{pointwise} bounds, it appears natural to lift the problem and regard \eqref{eq:basic-fm} as an inequality between ``absolute value'' of vectors, just like in \eqref{eq:musterfilmay}.
In this note we will extend the ideas in~\cite{Ber12,FilMay12} to the much more general setting of operators on Banach lattices: a canonical setting where absolute values of vectors are well-defined.
This includes elliptic operators on Lebesgue spaces or spaces of continuous functions vanishing at infinity over a locally compact metric space, but also more exotic objects, like pseudo-differential operators (including fractional Laplacians or Dirichlet-to-Neumann operators), infinite dimensional Ornstein--Uhlenbeck processes, or Lindbladian generators on spaces of Schatten--von Neumann operators over a  Hilbert space. 

\autoref{thm:filmay-very-abs} especially applies to large classes of (not necessarily self-adjoint) uniformly elliptic and even degenerate second order operators satisfying a weak maximum principle,
and in particular to Schrödinger operators with a positive (possibly singular) potential and/or with boundary conditions on bounded open domains of $\R^d$ or on suitable subset of compact manifolds -- this is the setting discussed in~\cite{FilMay12,ArnDavFil19} -- but also on finite metric \cite{HarMal20,MugPlu23} or combinatorial graphs \cite{FilMayTao21}: in these settings, again, $\rho=\mathbf{1}$ is the usual choice.

 The scope of this note is twofold. On the one hand, we borrow different \textit{Ansätze} proposed by several authors to sharpen the original bound by Filoche--Mayboroda--van den Berg and unify them: we can thus present, in an abstract functional analytical framework, a minimal set of assumptions that allow to re-derive numerous landscape-functions-based estimates for differential and difference operators  that have been presented by different authors over the last ten years. 
On the other hand, we are going to elaborate on \autoref{thm:filmay-very-abs}  in three different directions.
 
Firstly, we extend in  Section~\ref{sec:torsland-nonl} the scope of the theory of landscape functions to nonlinear homogeneous operators, whose spectral theory  relies on well-known variational tools, cf.~\cite[Chapter~III]{FucNecSou73}.

Secondly, we discuss in Section~\ref{sec:torsland-lin} two different linear setups to which the ideas of van den Berg and  Filoche--Mayboroda can be extended: we consider the cases of landscape functions for the eigenvectors of operators satisfying some very weak form of a maximum principle (namely, existence of a positive operator that dominates the inverse), and of integral operators (in the sense of~\cite[Definition~1.1]{AreBuk94}), in Section~\ref{sec:orderprop} and~\ref{sec:integr}, respectively.
We present sharper versions of \autoref{thm:filmay-very-abs} and argue that the key argument in their proof is the assumption that $-A$ generates a positive semigroup, or even a semigroup that is merely dominated by a positive one. In passing, we also complement~\eqref{eq:basic-fm} with a similar upper estimate in terms of a \textit{parabolic landscape function}: this corresponds to replacing the Green function by the heat kernel in the  estimates in~\cite{Ber12,FilMay12}; in the case of Schrödinger operators, this idea has already been explored in \cite[Theorem 2 and Corollary]{HosQuaSte22}.

Thirdly, we show in Section~\ref{sec:eigenv-est} that landscape function methods can be adapted to derive lower bounds on the principal eigenvalue of the relevant (linear or nonlinear)
  operators: we here elaborate on an old idea by Donsker--Varadhan~\cite{DonVar75,DonVar76} (as paraphrased e.g.\ in~\cite[Lemma~2.1]{BovGayKle05}), later rediscovered
e.g.\ in~\cite{BanCar94} and, for nonlinear homogeneous operators, in~\cite{GioSmi10}: observe that these results predate the development of the theory of landscape functions. We describe a whole class $\mathcal Q^A_+(E)$ of landscape functions -- including the torsion function $v:=A^{-1}\mathbf{1}$ -- and minimize their gauge norm within $\mathcal Q^A_+(E)$ to derive  an improved lower bound. 

We conclude our note reviewing in Section~\ref{sec:appl} several classes of operators to which our results can be applied: In Section~\ref{sec:toy-model} we will show how our techniques can be used to derive rather sharp estimates on the ground state of a self-adjoint operator; we will treat a nonlinear counterpart in Section~\ref{exa:main-nonlinear} and show that our methods can be seen as a relaxation of the search for the Cheeger constant. 
Finally, in Section~\ref{sec:magn-schr-stein} we present an application to a family of operators that may fail to have positive inverse: we derive from the Kato--Simon diamagnetic inequality 
a new proof of an estimate recently obtained in~\cite{HosQuaSte22} under somewhat stronger assumptions. 

\medskip

In the following we will assume that the reader is familiar with the theory of Banach lattices and of linear operators thereon, as presented e.g.\ in~\cite{Sch74,Nag86}: among other things, we will throughout use properties of positive linear (and order preserving nonlinear) operators, and sometimes assume that operators \textit{are dominated by other operators}, or that they even \textit{have a modulus operator}. 

An especially important notion is the following: Given a Banach lattice $E$ with positive cone $E_+$ and some $\rho\in E_+$, the \emph{principal ideal generated by $\rho$} is the vector space
\[
E_\rho:=\{f\in E: \exists c\ge 0 \hbox{ s.t. }|f|\le c\rho\}
\]
which by the results in \cite[Section~II.7]{Sch74} becomes a Banach lattice when endowed with the \emph{gauge norm}
\begin{equation}\label{eq:min-gauge}
\|f\|_\rho:=\inf \{c\ge 0:|f|\le c\rho\}.
\end{equation}
The easiest possible -- and perhaps canonical -- example is that of Lebesgue spaces over a finite measure space $X$. Then $E=L^p(X)$ is for each $p\in [1,\infty]$ a Banach lattice (with respect to the pointwise defined order relation between functions); upon taking $\rho=\mathbf{1}$, one e.g.\ finds $E_\rho=L^\infty(X)$.

More generally, because $E_\rho$ is in fact an AM space -- indeed, it is isometrically Banach lattice isomorphic to $C(K;\R)$ for some compact Hausdorff space $K$ (with $\rho\simeq \mathbf{1}_K$) -- the infimum in~\eqref{eq:min-gauge} is attained, i.e., $|f|\le \|f\|_\rho \rho$ for all $f\in E_\rho$. 
We will also make use of ideas from the theory of \textit{eventually positive $C_0$-semigroups}, see~\cite{Glu22} for a gentle introduction.

{
\section{Eigenvector bounds by landscape functions: The nonlinear case}\label{sec:torsland-nonl}

The proof of \autoref{thm:filmay-very-abs} relies solely on the homogeneity of $A^{-1}$, rather than on the full linearity of $A^{-1}$. We will elaborate on this idea in this section, often making good use of the properties of (nonlinear) maximal monotone operators:
two standard references are \cite{Bre73} for the general theory of such operators, and \cite{BenCra91,CipGri03} for their interplay with the theory of Banach lattices.

Given some $p\ge 1$, a vector space $H$ and a (possibly nonlinear and multi-valued) operator $A$ on $H$,
we follow~\cite{BunBur20} and say that $A$ is \emph{$(p-1)$-homogeneous} if its domain $D(A)$ is a cone and
\begin{equation}\label{eq:homogen}
A(cf)=\sign c |c|^{p-1}Af\qquad \hbox{for all }f\in D(A)\hbox{ and all }c\ne 0.
\end{equation}
A direct computation yields the following.

\begin{lemma}\label{lem:homog-inv}
Let $A$ be a $(p-1)$-homogeneous operator on a vector space $H$, for some $p\ge 1$.
Then its inverse -- if it exists -- is $(p-1)^{-1}$-homogeneous, i.e., it satisfies
\begin{equation}\label{eq:homogen-inv}
A^{-1}(cf)=\sign c |c|^\frac{1}{p-1}A^{-1}f\qquad \hbox{for all }f\in H\hbox{ and all }c\ne 0.
\end{equation}
\end{lemma}

\begin{exa}
(1) Linear operators are $1$-homogeneous, and so is for all $p\in[1,\infty)$, the \textit{normalized $p$-Laplacian} $\Delta^{\mathrm{N}}_p$, see~\cite{Kaw20} and references therein, defined by
\[
\Delta_p^{\mathrm{N}}:f\mapsto \frac{1}{p}|\nabla f|^{2-p}\nabla\cdot (|\nabla f|^{p-2}\nabla f)
\]
as well as 
\[
f\mapsto \frac{1}{p}|f |^{2-p}\nabla \cdot (|\nabla f|^{p-2}\nabla f).
\]

(2) A subdifferential $\partial\mathcal E$ of  a convex, proper, lower-semicontinuous functional $\mathcal E$ on a Hilbert space $H$ (see~\cite[Exemples~2.1.4 and 2.3.4]{Bre73}) 
is a $(p-1)$-homogeneous operator whenever
$\mathcal E$ is absolutely $p$-homogeneous, i.e., provided
\[
\mathcal E(cf)=
\begin{cases}
|c|^p\mathcal E(f),\qquad &\hbox{for all }f\in H,\hbox{ if }c\ne 0,\\
0,&\hbox{if }c=0,
\end{cases}
\]
for any $p\in [1,\infty)$.
Thus, examples of absolutely $(p-1)$-homogeneous operators are given by the (standard) $p$-Laplacians on open bounded domains of $\R^d$ (in particular, with homogeneous Dirichlet or Neumann conditions, but also with suitable Robin-type conditions~\cite{Le06}), on combinatorial \cite{Mug13} or metric graphs \cite{DelRos16};  by Finsler $p$-Laplacians \cite{FerKaw09}; by non-local $p$-Laplacians \cite[Chapter~6]{AndMazRos10}; by $p$-Dirichlet-to-Neumann operators~\cite{ChiHauKen15}; by fractional $p$-Laplacians~\cite{DelGomVaz21}; or by subdifferential of general $p$-Cheeger energies on general metric measure spaces recently introduced in~\cite{GorMaz22}. Operators associated with porous medium equations in $H^{-1}(\Omega)$
are $(p-1)$-homogeneous, too, for suitable $p$.
\end{exa}

Recall that a single-valued operator $T$ on a vector lattice is said to be \textit{order preserving} if
\[
Tf\le Tg\qquad \hbox{for all }f,g\in E\hbox{ with }f\le g;
\]
and to \textit{dominate} another single-valued operator $S$ on $E$ if
\[
|Sf|\le T|f| \qquad\hbox{ for all }f\in E.
\]

The following property is well-known in the case of linear operators, see the proof of~\cite[Proposition~2.20]{Ouh05}. 

\begin{lemma}\label{lem:positi-charac}
Let $E$ be a vector lattice and let $A$ be a (possibly multi-valued) invertible $(p-1)$-homogeneous operator for some $p\ge 1$. If $A^{-1}$ is single-valued and order preserving, then it dominates itself.
In particular, $A^{-1}f\ge 0$ for all $f\ge 0$.
\end{lemma}

\begin{proof}
Because $-|f|\le \pm f\le |f|$, we deduce from~\eqref{eq:homogen-inv} that $|\pm A^{-1}f|=|A^{-1}(\pm f)|\le A^{-1}|f|$.
\end{proof}

If, for some $p\ge 1$, $A$ is a (possibly multi-valued) $(p-1)$-homogeneous operator with domain $D(A)$ and there exists $0\ne \varphi\in D(A)$ and $\lambda\in \R$ such that
\begin{equation}\label{eq:eigenv-incl}
A\varphi\ni \lambda|\varphi|^{p-2}\varphi,
\end{equation}
then $\varphi$ is said to be an \textit{eigenvector} of $A$ for the \textit{eigenvalue} $\lambda$ or, shortly, we call $(\lambda,\varphi)$ an \textit{eigenpair} of $A$. 

\begin{prop}\label{prop:landscape-abstr-nonlin}
Let $E$ be a real Banach lattice, and let $\rho\in E_+$.
Let $p\ge 1$, and let $A,B$ be (possibly multi-valued) maximal monotone, $(p-1)$-homogeneous  and invertible operators on $E$, and such that $B^{-1}$ is order preserving and dominates $A^{-1}$.
Then for any eigenpair $(\lambda,\varphi)$ of $A$ such that $\varphi\in E_\rho$ there holds
\begin{equation}\label{eq:landscape-superabstr-nonl-pos}
|\varphi|\le  
|\lambda|^\frac{1}{p-1} 
\left\|{\varphi}\right\|_\rho B^{-1}\rho.
\end{equation}
If $A$ has order preserving inverse, then we can take $A=B$ in \eqref{eq:landscape-superabstr-nonl-pos}.
\end{prop}

\begin{proof}
Applying the single-valued operator $A^{-1}$
 to both sides of~\eqref{eq:eigenv-incl} yields 
\[
|\varphi|=\left| A^{-1}(\lambda|\varphi|^{p-2}\varphi)\right|\\
\le
B^{-1}(|\lambda| |\varphi|^{p-1}).
\]
Now we apply \autoref{lem:homog-inv} and proceed like in the proof of  \autoref{thm:filmay-very-abs} to deduce the claim, using the fact that $E_\rho$ is an $AM$-space and, hence, $\| |\varphi|^{p-1}\|_\rho\le \| \varphi\|_\rho^{p-1}$.
The second assertion follows from \autoref{lem:positi-charac}.
\end{proof}

{
Inspired by the general setting proposed in~\cite{Ota84}, we may also consider eigenpairs of $(p-1)$-homogeneous operators with respect to another (single-valued!) $(p-1)$-homogeneous operator $\Psi$: by this we mean that
\[
A\varphi\ni \lambda \Psi\varphi.
\]
The case discussed in the first part of this section corresponds to the choice
\[
\Psi\varphi:=|\varphi|^{p-2}\varphi
\]
but if $(X;\mu)$ is a finite measure space and $E=L^q(X;\mu)$ we may also let
\[
\Psi\varphi:=\|\varphi\|^{p-2}_2 \varphi
\]
or, more generally,
\[
\Psi\varphi:=\|\varphi\|^{p-q}_q |\varphi|^{q-2} \varphi,\qquad 1\le q<\infty.
\]

We can hence obtain a generalization of \autoref{prop:landscape-abstr-nonlin}: the proof is identical and we omit it.

\begin{cor}
Let $E$ be a real Banach lattice, and let $\rho\in E_+$.
Let $p\ge 1$, and let $A,B$ be (possibly multi-valued) maximal monotone, $(p-1)$-homogeneous  and invertible operators on $E$, and such that $B^{-1}$ is order preserving and dominates $A^{-1}$.
Then for any eigenpair $(\lambda,\varphi)$ of $A$ with respect to a $(p-1)$-homogeneous operator $\Psi$ on $E$ we have
\[
|\varphi|\le |\lambda|^\frac{1}{p-1} \|\Psi \varphi\|_\rho^\frac{1}{p-1} B^{-1} \rho.
\]
\end{cor}

}

\section{Eigenvector bounds by landscape functions: The linear case}\label{sec:torsland-lin}

\subsection{Estimates by order properties}\label{sec:orderprop}
For the sake of later reference, let us first state a more general version of \autoref{thm:filmay-very-abs} that applies to operators that are not invertible: its proof is an obvious modification of~\eqref{eq:musterfilmay}.

\begin{theo}\label{thm:filmay-very-abs-lattice}
Let $E$ be a Banach lattice, and let $\rho\in E_+$. Let $A$ be a closed linear operator on a Banach lattice $E$ such that $\mu+A$ is invertible and $(\mu+A)^{-1}$ is dominated by some bounded linear operator $(\mu+A)^{(-1)}_+$. If $(\lambda,\varphi)$ is an eigenpair of $A$ such that $\varphi\in E_\rho$, then
\begin{equation}\label{eq:basic-fm-superabstr-lattice}
|\varphi|\le |\mu+\lambda| \|\varphi\|_\rho (\mu+A)^{(-1)}_+ \rho.
\end{equation}
Clearly, we may take $(\mu+A)^{(-1)}_+:=(\mu+A)^{-1}$ if $\mu+A$ has positive inverse.
\end{theo}

\begin{rem}
{
There is a curious counterpart of the concept of landscape function in the classical Birman--Schwinger theory: Given a potential $0\ge V\in L^\infty(\R^d)$, a simple manipulation shows that $(\lambda,\varphi)$ is an eigenpair for the Schrödinger operator $-\Delta+V$
if and only if 
\begin{equation}\label{eq:basic-bs}
(\lambda+\Delta)^{-1} V\varphi=
 \varphi,
\end{equation}
i.e., if and only if $(1,V^\frac12 \varphi)$ is an eigenpair for the operator $V^\frac12 (\lambda+\Delta)^{-1}V^\frac12 $, where formally $V^\frac{1}{2}:=\sgn V |V|^\frac{1}{2}$. Birman and Schwinger, and many authors after them elaborated on this observation to find spectral correspondences between $-\Delta+V$ and $V^\frac12 (\lambda+\Delta)^{-1}V^\frac12 $: see, e.g.,~\cite[Chapter~4]{LieSei10} for applications in modern mathematical physics. 

However, inspired by \eqref{eq:basic-bs} we can also find a new landscape function, 
hence deduce a new pointwise eigenvector bound: 
If $\mu> s(\Delta)$, then by the maximum principle $(\mu-\Delta)^{-1}$ is a positive operator and \eqref{eq:basic-bs} yields 
\begin{equation}\label{eq:basic-eigenestim-bs}
|\varphi| \le \|\varphi\|_\infty(\mu-\Delta)^{-1}|\lambda+\mu-V|\qquad\hbox{for all }\mu>s(\Delta),
\end{equation}
to be compared with
\begin{equation}\label{eq:basic-eigenestim-fm}
|\varphi|\le |\mu+\lambda| \|\varphi\|_\infty (\mu-\Delta+V)^{-1} \mathbf{1}\qquad\hbox{for all }\mu>s(\Delta-V),
\end{equation}
from \autoref{thm:filmay-very-abs-lattice}:
the relative size of $\lambda,V$ will decide which estimate is sharper.

(It should be appreciated that, in view of the Kato--Simon diamagnetic inequality, one also has the domination property $|(\mu+\Delta_a)^{-1}\psi|\le (\mu+\Delta)^{-1}|\psi|$ for any magnetic Laplacian $\Delta_a:=(i\nabla +a)^2$ (under very mild integrability assumptions on $a,V$, see~\cite{HunSim04,MelOuhRoz04}) and, hence, \eqref{eq:basic-eigenestim-bs} and \eqref{eq:basic-eigenestim-fm} also hold for all eigenpairs $(\lambda,\varphi)$ of 
$-\Delta_a+V$. We will elaborate on this in Section~\ref{sec:magn-schr-stein}.)
}
\end{rem}

Let us now present two further classes of landscape-type functions for possibly non-invertible operators. The first one is related to the  \textit{anti-maximum principle}, which has been studied for numerous linear and nonlinear operators since~\cite{ClePel79}:
roughly speaking, validity of a maximum/anti-maximum principle results in a sudden ``change of sign'' (from order preservation to order reversal) of an operator's resolvents 
 across the bottom of its spectrum. The second one replaces the order properties of resolvents by those of $C_0$-semigroups: we refer to~\cite{DanGluKen16b} for the terminology. Here and in the following, we denote by $s(T):=\sup\{\real \lambda: \lambda\in \sigma(T)\}$ the \textit{spectral bound}  of an operator $T$.

{

\begin{prop}\label{prop:landscape-abstr}
Let $A$ be a closed linear operator on a complex Banach lattice $E$, and let $\rho\in E_+$. 

Then the following assertions hold for each eigenpair $(\lambda,\varphi)$ of $A$ such that $\varphi\in E_\rho$.

\begin{enumerate}[(1)]

\item\label{item:amti}
Let $s(-A)\in \R$ be an eigenvalue and a pole of the resolvent $(\cdot+A)^{-1}$, and let the corresponding eigenprojector $P_*$ be \emph{strongly positive with respect to $\rho$}.
If $A$ is densely defined, then  both
\begin{equation}\label{eq:landscape-superabstr-pos-antim}
|\varphi|\le  |s(-A)+\varepsilon+\lambda|\|\varphi\|_\rho  (s(-A)+\varepsilon+A)^{-1}\rho
\end{equation}
and
\begin{equation}\label{eq:landscape-superabstr-neg}
|\varphi|\le  -|s(-A)-\varepsilon+\lambda|\|\varphi\|_\rho  (s(-A)-\varepsilon+	A)^{-1}\rho
\end{equation}
hold for all $\varepsilon>0$ sufficiently small.

\item\label{item:test} If $-A$ generates a  $C_0$-semigroup $(\e^{-t A})_{t\ge 0}$ that is \emph{individually eventually strongly positive} with respect to $\rho$, then  there is $t_0>0$ such that
\begin{equation}\label{eq:landscape-steiner}
|\varphi|\le \e^{t\real \lambda}\|\varphi\|_\rho  \e^{-tA} \rho\qquad\hbox{ for all }t\ge t_0.
\end{equation}
\end{enumerate}
\end{prop}

This suggests that $\e^{-tA}\rho$ is a landscape-like function: we will refer to it as a \textit{parabolic landscape function}. If $-A$ is a semigroup generator, and if $(\e^{-tA})_{t\ge 0}$ is individually eventually strongly positive with respect to $\rho$, then the condition that $P$ be strongly positive with respect to $\rho$ can be reformulated in terms of smoothing properties of $(\e^{-tA})_{t\ge 0}$, see~\cite[Theorem~5.2 and Corollary~5.3]{DanGluKen16b}.

\begin{proof}

(1) 
We apply~\cite[Theorem~4.4]{DanGluKen16b} and deduce that $(\mu+A)^{-1}$ is a strongly positive (resp., strongly negative) operator for all $\mu$ in a sufficiently small (depending on $\varphi,\rho$) left (resp., right) neighbourhood of $s(-A)$.
In particular, the estimates
\begin{equation}\label{eq:landscape-superabstr-pos-sharp}
\begin{split}
|\varphi|&\le  |s(-A)+\varepsilon+\lambda |  (s(-A)+\varepsilon+A)^{-1}|\varphi|\\
&\le  |s(-A)+\varepsilon+\lambda |\|\varphi\|_\rho  (s(-A)+\varepsilon+A)^{-1}\rho-c\rho
\end{split}
\end{equation}
and
\begin{equation}\label{eq:landscape-superabstr-neg-sharp}
\begin{split}
|\varphi|&\le -|s(-A)-\varepsilon+\lambda|  (s(-A)-\varepsilon+A)^{-1}|\varphi|\\
&\le -|s(-A)-\varepsilon+\lambda|\|\varphi\|_\rho  (s(-A)-\varepsilon+A)^{-1}\rho-c\rho
\end{split}
\end{equation}
hold for some $c>0$ and all $\varepsilon$ small enough.
In either estimate, the former inequality holds because $A^{-1}$ dominates (resp., anti-dominates) itself, while the latter follows from the inequality $|\varphi|\le \|\varphi\|_\rho \rho$.

(2) By the Spectral Mapping Theorem $(\e^{-\lambda t},\varphi)$ is an eigenpair of $\e^{-tA}$.
Hence, due to individual eventual positivity of $(\e^{\real \lambda t}\e^{-tA})_{t\ge 0}$, there is $t_0$ (depending on the positive vector $
\|\varphi \|_\rho   \rho-|\varphi|$) such that
\begin{equation}\label{eq:landscape-steiner-proof}
|\varphi|=|\e^{\lambda t}\e^{-tA}\varphi|\le \|\varphi\|_\rho   \e^{t\real \lambda}\e^{-tA} \rho\qquad\hbox{for  all }t\ge t_0.
\end{equation}
This concludes the proof.
\end{proof}

It is often possible to have $|(\mu+A)^{-1}\rho|$ dominated by another -- ideally, easy to compute -- landscape function: for instance, $(\mu+B)^{-1}\rho$, where $B$ is a different, simpler operator, or perhaps a different realization of the same operator with simpler boundary conditions. Likewise, further parabolic landscape functions can be found even whenever the semigroup $(\e^{-tA})_{t\ge 0}$ is not even \textit{eventually} positive. The following \autoref{cor:landscape-abstr-domin} allows for upper estimates of eigenvectors in terms of a whole ensemble of landscape functions and, in turn, for an interesting sharpening of~\eqref{eq:basic-fm-superabstr-lattice} for eigenvectors associated with higher eigenvalues, where the right hand side in~\eqref{eq:basic-fm-superabstr-lattice} blows up, see Section~\ref{sec:toy-model} below.  This idea already appears, at least implicitly, in \cite[Appendix]{LyrMayFil15} and \cite[Theorem~III.4]{BalLiaGal20}.

\begin{prop}\label{cor:landscape-abstr-domin}

Let $E$ be a Banach lattice, and let $\rho\in E_+$.
Let $A,B$ be generators of $C_0$-semigroups on $E$ such that $(\e^{-tA})_{t\ge 0}$ is dominated by $(\e^{-tB})_{t\ge 0}$.

Then each eigenpair $(\lambda,\varphi)$ of $A$ such that $\varphi\in E_\rho$ satisfies
\begin{equation}\label{eq:landscape-superabstr-pos-domin-resolv}
|\varphi|\le \|\varphi\|_\rho \inf_{\real \mu>\max\{s(-A),s(-B)\}}|\real (\lambda+ \mu)| (\real \mu+B)^{-1}\rho
\end{equation}
and
\begin{equation}\label{eq:landscape-steiner-domin-semig}
|\varphi|\le \|\varphi\|_\rho \inf_{t>0} \e^{t\real \lambda} \e^{-tB} \rho.
\end{equation}

We can take $A=B$ if $(\e^{-tA})_{t\ge 0}$ is positive: in particular, \begin{equation}\label{eq:landscape-bootstr}
|\varphi|\le  \|\varphi\|_\rho 
\inf_{\delta>0}\left(\real \lambda+s(-A)+\delta\right) \left[\left(s(-A)+\delta+A\right)^{-1}\rho\right]\quad\hbox{and}\quad |\varphi|\le \|\varphi\|_\rho \inf_{t>0} \e^{t\real \lambda} \e^{-tA} \rho.
\end{equation}
\end{prop}

Observe that even if both $(\mu+C)^{-1}\rho,\e^{-tC}\rho$ lie in $E_\rho$ for all $\mu\in\rho(C)$ and $t\ge 0$, their infima over $\mu,t$ generally need not.
All estimates in \eqref{eq:landscape-superabstr-pos-domin-resolv}, \eqref{eq:landscape-steiner-domin-semig}, and \eqref{eq:landscape-bootstr} are thus meant in a pointwise sense, using the identification $E_\rho\simeq C(K;\R)$ for some compact Hausdorff space by~\cite[Section~II.7]{Sch74}.

\begin{proof}
By the Spectral Mapping Theorem $(\e^{-\lambda t},\varphi)$ is for all $t\ge 0$ an eigenpair of $\e^{-tA}$. By domination and due to positivity of the dominating operator, this yields \eqref{eq:landscape-steiner-domin-semig} by observing that
\begin{equation}\label{eq:landscape-steiner-proof-domin}
\begin{split}
|\varphi|&=|\e^{\lambda t}\e^{-t A}\varphi|\le \e^{t \real \lambda }\e^{-t B}|\varphi|\le \e^{t\real \lambda}\|\varphi\|_\rho   \e^{-t B} \rho\qquad \hbox{for all }t\ge 0.
\end{split}
\end{equation}
Now,  using the representation of the resolvents as Laplace transforms of the $C_0$-semigroups one sees that $(\real \mu+B)^{-1}$ dominates $(\mu+A)^{-1}$, too, for all $\mu\in \C$ with $\real \mu>\max\{s(-A),s(-B)\}$: we thus immediately deduce that
\begin{equation*}\label{eq:landscape-superabstr-pos-domin-resolv-bis}
|\varphi|\le \|\varphi\|_\rho |\lambda+\mu| (\real \mu+B)^{-1}\rho
\end{equation*}
Finally, \eqref{eq:landscape-bootstr} follows taking the infimum over all $\mu$ : in particular, the term $ |\lambda+\mu|$ is minimized if we pick along $\real \mu+i\R$ the complex number with same imaginary part as $\lambda$.
 
The last assertion is obvious, since a positive semigroup dominates itself.
\end{proof}

In the spirit of \autoref{prop:landscape-abstr}.\ref{item:test}, a counterpart of \eqref{eq:landscape-steiner-domin-semig} may even be formulated in the case of mere \textit{individual eventual} domination, a phenomenon which has been recently described in \cite{GluMug21,AroGlu22c}. We omit the obvious details.

\begin{rem}\label{rem:altern-steiner}
If $A$ generates a positive, self-adjoint $C_0$-semigroup in $E=L^2(X)$, then we can express \eqref{eq:landscape-steiner-domin-semig} by saying that $\varphi$ can only localize in the sets $\{\frac{1}{\e^{-tA}\rho}\le \e^{t \lambda}\}=\{-\frac{\log (\e^{-tA}\rho)}{t}\le \lambda\}$, for any $t>0$. If, additionally, $A$ has positive inverse, then plugging (for $\mu=0$) \eqref{eq:landscape-superabstr-pos-domin-resolv} into the penultimate term in \eqref{eq:landscape-steiner-proof-domin} shows that 
\[
|\varphi|\le \|\varphi\|_\rho |\lambda|\e^{t \lambda }\e^{-tA} A^{-1}\rho\qquad \hbox{for all }t>0,
\]
i.e., $\varphi$ can only localize in $\{\frac{1}{A^{-1}\e^{-tA}\rho}\le |\lambda| \e^{t\real \lambda}\}$. It would be interesting to understand the relation with the alternative landscape function $\frac{\e^{-tA}\rho}{A^{-1}\rho}$ proposed in \cite[Section~4.2]{Ste21}.
\end{rem}

Let us note an  immediate consequences of \eqref{eq:landscape-superabstr-pos-domin-resolv}.

\begin{cor}\label{cor:heatkernel}
Let $X$ be a $\sigma$-finite measure space, and let $-A$ generate a positive, trace-class $C_0$-semigroups on $L^2(X)$.
Let additionally $A$ be a self-adjoint operator, and let $(\lambda_k,\varphi_k)_{k\in \N}$ be a sequence of eigenpairs of $A$ such that $(\varphi_k)_{k\in \N}$ form an orthonormal basis of $L^2(X)$. Then for all $0\le \rho\in L^2(X)$ the heat kernel $(k_t)_{t\ge 0}$ of $(\e^{-tA})_{t\ge 0}$ satisfies
\begin{equation}\label{eq:estim-heatkernel}
k^A_t(x,y)\le  \sum_{k\in\N} \e^{-t\lambda_k}  \tilde{\rho}_k(x)\tilde{\rho}_k(y),\qquad x,y\in X,
\end{equation}
provided the series on the right hand side converges in $L^2(X)$, and 
where
\[
\tilde{\rho}_k(z):=\|\varphi_k\|_\infty^2\inf_{\delta>0}|\lambda_k+s(-A)+\delta| \left[(s(-A)+\delta+A)^{-1}\rho\right](z),\qquad z\in X.
\]
\end{cor}

If, in particular, $A$ has positive inverse, then \eqref{eq:estim-heatkernel} reads
\begin{equation}\label{eq:filmay-heatkernel-invertible}
k^A_t(x,y)\le  \sum_{k\in\N} \e^{-t\lambda_k} \|\varphi_k\|_\infty^2 |\lambda_k|^2 A^{-1}\rho(x) A^{-1}\rho(y),\qquad x,y\in X,
\end{equation}
with the right hand side converging in $L^2(X)$ if, in particular, $(\|\varphi_k\|_\infty)_{k\in \N}$ grows at most polynomially in $k$. (Non-trivial cases where $(\|\varphi_k\|_\infty)_{k\in \N}$ is even bounded are known, see e.g.~\cite[Lemma~3]{BifKer24}.)
In this case, the ground state qualifies as a landscape function, see~\cite[Lemma~4.2.2 and Theorem~4.2.4]{Dav89}.

\begin{proof}
By Mercer's Theorem the heat kernel of $(\e^{-tA})_{t\ge 0}$ is given by
\begin{equation}\label{eq:heatk-mercer}
k^A_t(x,y)=\sum_{k=1}^\infty \e^{-t\lambda_k}\varphi_k(x)\varphi_k(y),\qquad x,y\in X.
\end{equation}
The assertion now follows immediately from
\eqref{eq:landscape-bootstr} upon taking $\rho=\mathbf{1}$.
\end{proof}

\begin{rem}\label{rem:modulusetc}
(1) Observe an immediate consequence of \autoref{cor:landscape-abstr-domin}, based on the notion of an operator's \textit{modulus}, i.e., its smallest positive dominating operator, see~\cite[Section~IV.1]{Sch74}: 
For each eigenpair $(\lambda,\varphi)$ of $A$ such that $\varphi\in E_\rho$ we have the following.
\begin{enumerate}[(1$'$)]
\item If the resolvent operators $(\zeta+A)^{-1}$ of $A$ have a modulus $[(\zeta+A)^{-1}]^\#$ for all $\zeta>s(-A)$, then 
\begin{equation}\label{eq:landscape-steiner-domin-modul-'}
|\varphi|\le |\real(\lambda+\mu)|\|\varphi\|_\rho  [(\real \mu+A)^{-1}]^\#\rho \qquad\hbox{for all }\mu\in\rho(A).
\end{equation}
\item If $A$ is a $C_0$-semigroup generator, and if $(\e^{-t A})_{t\ge 0}$ has a modulus $C_0$-semigroup $(\e^{-t A^\sharp})_{t\ge 0}$, then
\begin{equation}\label{eq:landscape-steiner-domin-semigr-'}
|\varphi|\le \|\varphi\|_\rho  \e^{t\real \lambda}\e^{-tA^\sharp} \rho\qquad\hbox{for  all }t\ge 0.
\end{equation}
\end{enumerate}
By definition of modulus operator, these bounds are the best possible in the class of those that can be derived in the framework of \autoref{cor:landscape-abstr-domin}.

However, a non-positive operator need generally not have a modulus operator; indeed, the existence of a modulus operator is known in only a few cases, including some classes of integral operators $T_q$, in which case the modulus of $T_q$ is the integral operator $T_{|q|}$: 
 this is the setting implicitly assumed in the second part of~\cite{FilMay12} and in~\cite[Proposition~0.1]{FilMay12b}.
Even in this case, there need not exist an operator whose resolvent kernel is $|q_\mu|$ (resp., whose heat kernel is $|p_t|$), which makes it difficult to compute the landscape functions $[(\mu+A)^{-1}]^\# \mathbf 1$ (resp., $\e^{-tA^\#}\mathbf{1}$).

(2) If, however, $A$ is a finite square matrix, or an infinite matrix that acts as a bounded operator on some $\ell^p$-space, then it is known (cf.~\cite[Example C.II-4.19]{Nag86}) that $(\e^{-tA})_{t\ge 0}$ does  have a modulus semigroup $(\e^{-tA^\sharp})_{t\ge 0}$, and that its generator $A^\sharp$ is given by $A^\sharp_{ii}=\real A_{ii}$ and $A^\sharp_{ij}=|A_{ij}|$ if $i\ne j$. If $s(-A^\sharp)<0$, then clearly $s(-A)<0$ and 
\[
|A^{-1}|\le \int_0^\infty |\e^{-tA}|\ud t\le \int_0^\infty \e^{-tA^\sharp }\ud t\le (A^\sharp)^{-1}:
\]
this sharpens the main result in \cite{LemPacOvd20}, where a similar but weaker bound was proved under additional assumptions.

(3) Let us stress a special case, for which we could not find a reference in the literature: we will follow the notation from~\cite[Section~2.1]{Mug14} and~\cite[Section~2]{Mug13}. Let $\mG=(\mV,\mE,\nu,\mu)$ be an undirected, uniformly locally finite combinatorial graph, with edge weight $\mu$ and vertex weight $\nu$, and let $(\alpha_{\mv\mw})_{\mv,\mw\in\mV}\subset \R$ such that $\alpha_{\mv\mw}=-\alpha_{\mw\mv}$ 
and $\alpha_{\mv\mw}=0$ whenever $\mv\not\sim\mw$: such $\alpha$ are called \textit{magnetic signatures} of $\mG$. Let $\mathcal L_\alpha$ denote 
the magnetic Laplacian on $\mG$, a bounded operator defined as the self-adjoint operator associated with 
the quadratic form
\[
f\mapsto \frac{1}{2}\sum_{\mv\sim\mw} \mu_{\mv\mw} |f(\mv)-\e^{i\alpha_{\mv\mw}}f(\mw)|^2,\qquad f\in \ell^2(\mV,\nu):
\]
this class of operators has been introduced in~\cite{ColTorTru11b} and includes the special cases of (non-magnetic) standard Laplacian $\mathcal L$ and signless Laplacian $\mathcal Q$ for $\alpha= 0$ and $\alpha=\pi$, respectively. By the result mentioned in (2), $(\e^{-t\mathcal L_\alpha})_{t\ge 0}$ has a modulus semigroup, viz $(\e^{-t\mathcal L})_{t\ge 0}$, i.e., $(-\mathcal L_\alpha)^\sharp=-\mathcal L$ for all magnetic signatures $\alpha$; see~\cite{LenSchWir21} for a domination result for the Friedrichs realisation of the magnetic Laplacian on more general, not uniformly locally finite graphs. On the other hand, $(\e^{t\mathcal L_\alpha})_{t\ge 0}$ is generally not positive: its modulus semigroup is $(\e^{t\mathcal Q})_{t\ge 0}$, i.e., $(\mathcal L_\alpha)^\sharp =\mathcal Q$ for all magnetic signatures $\alpha$.  In view of \autoref{cor:landscape-abstr-domin}, this is remarkable because, on \textit{regular} graphs, the low-energy eigenfunctions for $\mathcal L$ are clearly the high-energy eigenfunctions for $\mathcal Q$ -- the crucial observation at the core of the theory of so-called \textit{dual landscapes} as in~\cite{LyrMayFil15,WanZha21}! These discrete diamagnetic inequalities will help prove \autoref{prop:magn-land-appl-discr}.
\end{rem}

\begin{rem}\label{rem:irred}
(1) \autoref{thm:filmay-very-abs-lattice} can be applied to operators $A$ whose resolvents $(\mu+A)^{-1}$ are positive only for specific values $\mu$; and, in particular, such that $(\e^{-tA})_{t\ge 0}$ is \textit{not} positive. 
It has been observed in the last few years that this is not an uncommon phenomenon, see~\cite{AroGlu22c} and references therein.	

 For open bounded domains $\Omega\subset\R^d$,  an example is given by the bi-Laplacian $A:=\Delta^2$ with clamped boundary conditions
\[
u(z)=\frac{\partial u}{\partial \nu}(z)=0\qquad\hbox{for all }z\in \partial\Omega,
\]
which -- if $\Omega$ is close enough to being a ball -- has positive inverse by a famous result due to Boggio~\cite{Bog05}, even though $(\e^{-tA})_{t\ge 0}$ fails to be positive for any  $\Omega$ by \cite[Theorem~2.7]{Ouh05} and general properties of Sobolev spaces. 

(2) 
The elementary rescaling used in the proof of \eqref{eq:landscape-superabstr-pos-domin-resolv} and~\eqref{eq:landscape-bootstr} carries over to the case of nonlinear operators that are 1-homogeneous (but not to general $p$-homogeneous operators for $p>1$!).

(3) Let $-A$ generate a compact, positive, irreducible $C_0$-semigroup and denote by $(\lambda_*,\varphi_*)$ its Perron eigenpair.
The inverse power method yields that $(\lambda_* A^{-1})^n$ converges uniformly, as $n\to\infty$, to the projector $P_*$ onto the eigenspace spanned by the Perron eigenvector.
If, additionally, $A^{-1}$ is a sub-Markovian operator, and hence $A^{-1}\rho\le \rho$, then not only does the sequence of landscape functions $\left(A^{-n}\rho\right)_{n\in \N}$ converge to an eigenvector for the principal eigenvalue; but it also does so in a monotonically  decreasing way. 

\end{rem}

We conclude this section by extending the scope of the lower estimate in~\cite[Proposition~3.2]{ArnDavFil19} and \cite[Lemma~2.12]{WanZha21}. Recall that, by standard Kre\u{\i}n--Rutman theory, an operator $-A$ has a unique strictly positive normalized eigenvector -- the \emph{Perron eigenvector} $\varphi_*$ -- provided $-A$ generates a positive, irreducible, compact $C_0$-semigroup. 

\begin{cor}\label{cor:fil-may-plus}
Let $E$ be a Banach lattice, and let $\rho\in E_+$.
Let $A$ be an invertible linear operator on $E$. Then the following assertions hold.
\begin{enumerate}[(1)]
\item Let the inverse of $A$ be positive.
If $\rho\in D(A)$ and $0\ne A\rho\in E_\rho$,  then 
\[
A^{-1}\rho \ge \frac{\rho}{\|A\rho\|_\rho}.
\]

\item Let $-A$ generate a positive, irreducible, compact $C_0$-semigroup $(\e^{-t A})_{t\ge 0}$.  If $\varphi_*\in E_\rho$, then the Perron eigenpair  $(\lambda_*,\varphi_*)$ satisfies
\begin{equation}\label{eq:landscape-superabstr-perron}
0<\varphi_*< |\lambda_*|\|\varphi_*\|_\rho  A^{-1}\rho.
\end{equation}
If, additionally, $(\e^{-t A})_{t\ge 0}$ is analytic, then 
\begin{equation}\label{eq:landscape-steiner-perron}
0<\varphi_*< \e^{t \lambda_*}\|\varphi_*\|_\rho  \e^{-tA} \rho\qquad\hbox{for all }t\ge 0,
\end{equation}
and also
\begin{equation}\label{eq:balexpgro}
\varphi_*\le \|\varphi_*\|_\rho P_* \rho.
\end{equation}
\end{enumerate}
\end{cor}

\begin{proof}
(1) Observe that
\(
A\rho\le \|A\rho\|_\rho\rho,
\)
whence the claim follows, again by positivity of $A^{-1}$.

(2) By \cite[Proposition~C-III.3.5]{Nag86} the spectral radius of $A$ is a simple eigenvalue $\lambda_*$ and the associated eigenspace is spanned by a positive eigenvector $\varphi_*$. Furthermore, the upper inequality in~\eqref{eq:landscape-superabstr-perron} is strict by~\cite[Definition~C-III.3.1]{Nag86}; while \eqref{eq:landscape-steiner-perron} follows from~\cite[Theorem~C-III.3.2.(a)]{Nag86} if $(\e^{-tA})_{t\ge 0}$ is analytic. 
Finally, \eqref{eq:balexpgro} follows from  balanced exponential growth, i.e., $\lim_{t\to \infty}\e^{t\lambda_*}\e^{-tA}=P_*$, see~\cite[Exercises~V.3.9.(3)]{EngNag00}.
\end{proof}

\subsection{Estimates by integral kernels}\label{sec:integr}

Linear operators acting on function spaces $L^p(X)$ are integral operators whenever $X$ is a discrete space.
More interestingly,  each operator from $L^p(X)$ to $L^\infty(X)$, for any $p\in [1,\infty)$, is an integral operator;
this and further conditions for an operator's resolvent to be integral, or for it to generate a $C_0$-semigroup consisting of integral operators, are presented in~\cite[Section~4]{AreBuk94}. 

\begin{prop}\label{lem:fundam-hoeld}
Let $A$ be a closed linear operator on $C_0(X)$, where $X$ is a locally compact metric space; or else on $L^q(X)$,  for some $q\in [1,\infty]$ and some measure space $X$. Furthermore, let $\rho$ be a  function  on $X$. Then the following assertions hold for all eigenpairs $(\lambda,\varphi)$ of $A$.

\begin{enumerate}[(1)]
\item Let, for some $\mu\in \C$, $\mu+A$ be invertible and let its inverse be an integral operator with integral kernel $q_\mu$. Then for all $r_1,r_2\in [1,\infty]$ such that $r_1^{-1}+r_2^{-1}=1$ we have
\begin{equation}\label{eq:integr-k-estim-1}
|\varphi(x)|\le |\lambda+\mu| \left\|\frac{\varphi}{\rho} \right\|_{r_1}  \left(\int_X |q_\mu(x,y)\rho(y)|^{r_2}\ud y\right)^\frac{1}{r_2}\qquad \hbox{for all/a.e. }x\in X,
\end{equation}
provided both terms on the right hand side are finite.
\item Let $-A$ generate a $C_0$-semigroup and let, for some $t_0>0$, $\e^{t_0 A}$ be an integral operator with integral kernel $p_{t_0}$. Then for all $r_1,r_2\in [1,\infty]$ such that $r_1^{-1}+r_2^{-1}=1$ we have
\begin{equation}\label{eq:integr-k-estim-2}
|\varphi(x)|\le \e^{t_0\real \lambda} \left\|\frac{\varphi}{\rho} \right\|_{r_1} \left(\int_X |p_{t_0}(x,y)\rho(y)|^{r_2}\ud y\right)^\frac{1}{r_2}\qquad \hbox{for all/a.e. }x\in X,
\end{equation}
provided both terms on the right hand side are finite.
\end{enumerate}
\end{prop}

\begin{proof}
Using  Hölder's inequality, the claims follow directly from the relations \(\varphi=(\lambda+\mu) (\mu+A)^{-1}\varphi\) and \(\e^{-t\lambda }\varphi=\e^{-tA}\varphi\), respectively.
\end{proof}

\begin{exa}
(1) The fundamental estimate~\cite[(7)]{FilMay12}, which already in the original paper by  Filoche--Mayboroda extends the theory of landscape functions to operators whose Green function is not positive, is a special case of \eqref{eq:integr-k-estim-1} with $r_1=\infty$, $r_2=1$, for  $\mu=0$ and $\rho=\mathbf{1}$; whereas on a bounded, open, connected domain $\Omega\subset \R^d$, for the same choice of parameters, \eqref{eq:berbutt-gener} boils down to the known estimate $1\le \lambda_* \|v\|_\infty$, on the principal eigenvalue $\lambda_*$ of the elliptic operator with drift $A:=-\Delta^{\mathrm{D}}+B\cdot \nabla$ with Dirichlet boundary conditions~\cite{DonVar76}; here $v:=A^{-1}\mathbf{1}$.

(2) Let us now consider the fourth derivative with hinged boundary conditions, i.e., $A=(-\Delta^{\mathrm{D}})^2$, on $L^2(0,1)$. Then $\lambda_*=\pi^4\approx 97.40909$, whereas $A^{-1}\mathbf{1}(x)=\frac{1}{24}x^4 - \frac{1}{12}x^3 + \frac{1}{24}x$, whence
$\|A^{-1}\mathbf{1}\|_\infty^{-1}=76.8$. In the case of the forth derivative with clamped boundary conditions, see \autoref{rem:irred}.(1), we obtain $A^{-1}\mathbf{1}(x)=\frac{1}{24}x^4 - \frac{1}{12}x^3 + \frac{1}{24}x^2$, whence $\|A^{-1}\mathbf{1}\|_\infty^{-1}=384\le \lambda_*$.
\end{exa}

\section{Eigenvalue estimates, the torsion function, and alternative landscape functions}\label{sec:eigenv-est}

In \cite[Lemma~8.16]{Bovden16}, a classical result by Donsker and Varadhan \cite{DonVar75,DonVar76} was paraphrased as follows: 	If $A:=-\Delta-b\cdot \nabla$ is an  elliptic operator on a bounded open domain $\Omega$ -- say, with  Dirichlet conditions imposed on $\partial \Omega$ smooth enough, with uniformly elliptic leading term and an appropriate drift term $b$ --, then its principal eigenvalue $\lambda_*$ satisfies
\begin{equation}\label{eq:donvar17}
1\le \lambda_* \|-A^{-1}\mathbf{ 1}\|_\infty.
\end{equation}
An alternative proof that makes use of the maximum principle only is presented in \cite{LuSte17}:  this paves the road to extending the Donsker--Varadhan estimate to any  general elliptic operator that is defined weakly as the operator associated with the form
\[
a(f,g):=\int_\Omega \left(a \nabla f\cdot \nabla g+(b\cdot \nabla f)g+f(c\cdot \nabla g)+Vfg\right)\ud x
\]
with real-valued $L^\infty$-coefficients $a,b,c,V$ and Dirichlet or (dissipative) Robin boundary conditions, as the latter generate a positive semigroup in view of  the results in~\cite[Section~4.2]{Ouh05}, and, hence, their (invertible) generators enjoy a weak maximum principle.

In the spirit of the previous sections, let us extend \eqref{eq:donvar17} to more general settings.
 Recall that given a Banach lattice $E$, some $\rho\in E_+$ is called \textit{quasi-interior} if $E_\rho$ is dense in $E$. We denote by ${\mathcal Q}^A_+(E)$ the set of all quasi-interior points $\rho$ of $E$ such that $D(A)\subset E_\rho$.

\begin{prop}\label{cor:pseudo-giosmi-linear}
Let $A$ be a closed linear operator on a complex Banach lattice $E$. Let $\rho\in E_+$, and let $D(A)\subset E_\rho$. Then the following assertions hold.

\begin{enumerate}[(1)]
\item 
Assume that $\mu+A$ is invertible for some $\mu\in \C$, and that there exists an operator $A^{(-1)}_{+,\mu}$ that dominates its inverse. Then each approximate eigenvalue $\lambda$ of $A$ satisfies
\begin{equation}\label{eq:giosmi-approx-general-individ}
1\le \inf_{\rho\in {\mathcal Q}^A_+(E)}  |\lambda+\mu|  \|A^{(-1)}_{+,\mu} \rho\|_\rho.
\end{equation}

\item 
If, in particular, $A$ is resolvent-positive, i.e., $\zeta+A$ is positive for all $\zeta>s(-A)$, then  each approximate eigenvalue $\lambda$ of $A$ satisfies
\begin{equation}\label{eq:giosmi-approx-general-individ-posit-ur}
1\le \inf_{\rho\in {\mathcal Q}^A_+(E)}  \inf_{\real \mu>s(-A)} |\real(\lambda+\mu)|  \|(\real\mu+A)^{-1}\rho\|_\rho.
\end{equation}

\item If $-A$ generates a positive $C_0$-semigroup that is dominated by a further semigroup, say $(\e^{-tB})_{t\ge 0}$, then each eigenvalue $\lambda$ of $A$ satisfies
\begin{equation}\label{eq:giosmi-approx-general-parab-domin}
-\inf_{\rho\in {\mathcal Q}^A_+(E)} \inf_{t>0} \frac{\log\|\e^{-tB} \rho\|_\rho}{t}\le \real \lambda .
\end{equation}
and in particular
\begin{equation}\label{eq:giosmi-approx-general-parab}
-\inf_{\rho\in {\mathcal Q}^A_+(E)} \inf_{t>0} \frac{\log\|\e^{-tA} \rho\|_\rho}{t}\le \real \lambda .
\end{equation}
if $(\e^{tA})_{t\ge 0}$ is positive.
\end{enumerate}
\end{prop}

\begin{proof}
(1) To begin with, observe that if $\rho\in E_+$ and if $D(A)\subset E_\rho$, then  $A^{-1}$ is a bounded operator from $E$ to $E_\rho$, by the closed graph theorem. 
Let $\lambda$ be an approximate eigenvalue: we can, hence, consider a sequence $(\varphi_n)_{n\in\N}\subset D(A)$ with $\|\varphi_n\|\equiv 1$ and such that $r_n:=A\varphi_n-\lambda \varphi_n$ is a sequence that converges to 0; accordingly, $A^{-1 }r_n$ converges to 0 in $E_\rho$. Furthermore, $\varphi_n\in E_\rho$ and 
\[
\|\varphi_n\|_\rho \le C\|A\varphi_n\|\le C\left(\|r_n\|+|\lambda|\|\varphi_n\|\right)=C\left(\sup_{n\in\N}\|r_n\|+|\lambda|\right):=\tilde{C},
\]
i.e., $\left(\|\varphi_n\|_\rho\right)_{n\in\N}$ is bounded; also, $\|\varphi_n\|_\rho\ge c\|\varphi_n\|\equiv c$ because  $E_\rho$ is continuously embedded into $E$, hence $\left(\|\varphi_n\|_\rho^{-1}\right)_{n\in\N}$ is bounded, too.
 Now,
\[
\varphi_n=(\lambda+\mu)(\mu+ A)^{-1}\varphi_n +A^{-1}r_n\qquad \hbox{for all }n\in \N,
\]
and therefore
\[
|\varphi_n|\le (\lambda+\mu)|(\mu+ A)^{-1}\varphi_n| +|A^{-1}r_n|
 \le (\lambda+\mu)|A^{(-1)}_{+,\mu}\varphi_n| +|A^{-1}r_n| \qquad \hbox{for all }n\in \N,
\]
whence
\[
1\le |\lambda+\mu| \|(\mu+ A)^{(-1)}_+ \rho\|_\rho +\|\varphi_n\|_\rho^{-1}\| A^{-1}r_n\|_\rho \qquad \hbox{for all }n\in \N.
\]
Now, \eqref{eq:giosmi-approx-general-individ} can be deduced passing to the limit.

(2)
 \eqref{eq:giosmi-approx-general-individ-posit-ur} follows likewise, because $(\real \mu+A)^{-1}$ dominates $(\mu+A)^{-1}$ for all $\mu\in \C$ with $\real \mu>s(-A)$.

(2) If $(\e^{-tA})_{t\ge 0}$ is positive, then \eqref{eq:landscape-steiner} holds for all $t\ge 0$: taking the gauge norm $\|\cdot\|_\rho$ and then logarithms of both sides we find
\[
-t\real \lambda\le \log \|\e^{-tA}\rho\|_\rho\qquad \hbox{for all }t\ge 0,
\]
whence the claim follows.
 \end{proof}
 
 \begin{rem}
In the prototypical case of Schrödinger operators $A:=-\Delta+V$ with Dirichlet conditions on a bounded open domain $\Omega\subset \R^d$, possible examples of general landscape functions are given by $v:=(-\Delta+V)^{-1}\varphi^{(\beta)}_*$, where $\varphi^{(\beta)}_*$ is the ground state of the free Laplacian with Robin  conditions with parameter $\beta$: a special case is, in particular, $\varphi^{(\beta=0)}_*=\mathbf{1}\in {\mathcal Q}^A_+(L^2(\Omega))$, which \textit{a posteriori} justifies the usual choice of the canonical torsion function as landscape function.
\end{rem}

\begin{rem}
Eigenvalue estimates can, of course, also be derived from the landscape functions for integral operators from Section~\ref{sec:integr}. Indeed, under the assumption of~\autoref{lem:fundam-hoeld}, let $\rho=\mathbf 1$ be an admissible choice. We then obtain the following special cases upon taking $\|\cdot\|_{r_1}$-norms of both sides of \eqref{eq:integr-k-estim-1} and~\eqref{eq:integr-k-estim-2}:
\begin{itemize}
\item with $r_1=\infty$ and $r_2=1$, and provided $\varphi\in L^\infty(X)$,
\begin{equation}\label{eq:berbutt-gener}
1\le |\lambda+\mu| \sup_{x\in X}\int_X |q_\mu(x,y)|\ud y\quad\hbox{and}\quad -\frac{1}{t_0}\log \sup_{x\in X}\int_X |p_{t_0}(x,y)|\ud y \le \real \lambda;
\end{equation}
\item with $r_2=2$ and $r_2=2$, and provided $\varphi\in L^2(X)$,
\[
1\le |\lambda+\mu| \|q_\mu\|_{L^2}\quad\hbox{and}\quad 
-\frac{1}{t_0}\log \|p_{t_0}\|_{L^2} \le \real \lambda
\]
(the latter is, however, rougher than the classical estimate
\(-\inf_{t>0}\frac{1}{t}\log \|e^{-tA}\|\le -s(-A)\),
since $\|p_{t_0}\|_{L^2}$ is the Hilbert--Schmidt norm of $\e^{-t_0 A}$);
\item with $r_1=1$ and $r_2=\infty$, and provided $\varphi\in L^1(X)$,
\[
1\le |\lambda+\mu| \sup_{y\in X}\int_X |q_\mu(x,y)|\ud x\quad\hbox{and}\quad -\frac{1}{t_0}\log \sup_{y\in X}\int_X |p_{t_0}(x,y)|\ud x \le \real \lambda.
\]
\end{itemize}
\end{rem}
 
Likewise, taking the $\|\cdot\|_\rho$-norm of both sides of \eqref{eq:landscape-superabstr-nonl-pos} we immediately obtain a nonlinear version of \autoref{cor:pseudo-giosmi-linear} for the setup introduced in Section~\ref{sec:torsland-nonl}.

\begin{prop}\label{cor:spectral-estim}
Let $A$ be a maximal monotone operator on a real Banach lattice $E$ that is $(p-1)$-homogeneous, for some $p\ge 1$, and invertible with order preserving inverse.
Then each eigenvalue $\lambda$ of $A$ satisfies 
\begin{equation}\label{eq:giosmi-nonlin-abstr}
\sup_{\rho\in {\mathcal Q}^A_+(E)} \|A^{-1}\rho\|_\rho^{1-p} \le  |\lambda|.
\end{equation}
\end{prop}

What kind of $\rho\in {\mathcal Q}_+^A(E)$ should be taken if we try to optimize  the above (approximate) eigenvalue estimates?
A natural choice is $\rho=\mathbf 1$, whenever $E=C(X)$ or $E=L^q(X)$ for some \textit{compact} metric space $X$ or some \textit{finite} measure space $X$, respectively, and $A$ is an operator whose ground states are \textit{not} constant; in the relevant case where $-A$ is the Laplacian with \textit{Dirichlet} conditions, $\mathbf{1}$ is the ground state for the Neumann realization of the same operator.

The following shows that, more generally, $A^{-1}\varphi^B_*$ is often an admissible landscape function, whenever $\varphi^B_*$ is the Perron eigenvector of an operator $B$ such that $(\e^{-tB})_{t\ge 0}$ dominates $(\e^{-tA})_{t\ge 0}$: this is, e.g., the case if $A$ is a Schrödinger operator with positive potential with Dirichlet conditions, and $B$ is the free Laplacian with Robin conditions with positive parameter  -- and, in particular, the Laplacian with Neumann conditions, i.e., $\varphi^*_B=\mathbf{1}$.

\begin{lemma}
Let $A$ generate a $C_0$-semigroup on a Banach lattice $E$, and let 
$B$ generate a compact, positive, irreducible $C_0$-semigroups $(\e^{-tB})_{t\ge 0}$ on $E$ that eventually dominates $(\e^{-tA})_{t\ge 0}$. 

Let $\varphi^B_*$ denote the Perron eigenvector of $B$.
If $(\e^{tB})_{t\ge 0}$ is eventually differentiable and $D(B^k)\subset E_{\varphi^B_*}$ for some $k\in \N$, then each eigenvector of $A$ lies in $E_{\varphi^B_*}$.
\end{lemma}

\begin{proof}
Let $(\lambda^A,\varphi^A)$ be an eigenpair of $A$, and let $t$ be large enough that $\e^{tA}$ is dominated by $\e^{tB}$ and $\e^{t B}$ maps $E$ to $D(B^k)$, cf.\ the discussion after~\cite[Definition~II.4.3]{EngNag00}. We then have
\[
|\varphi^A|=|\e^{t\lambda^A}\e^{-tA}\varphi^A|\le \e^{t\real \lambda^A}\e^{-tB}|\varphi^A|\le \e^{t\real \lambda^A}\|\e^{-tB}|\varphi^A| \|_{\varphi^B_*}\varphi^B_*
\]
using the fact that, by assumption, $\e^{tB}E\subset E_{\varphi^B_*}$.
\end{proof}

\section{Applications}\label{sec:appl}

If we take $E=L^q(X)$ for some finite measure space $X$ and some $q\in [1,\infty]$, then $\rho=\mathbf{1}$ is a quasi-interior point of $E$ and we end up with $E_\rho=E_{\mathbf{1}}=L^\infty(X)\hookrightarrow E$. This simple but illustrative example will be the basic setting discussed in the following examples. All plots presented in this section are obtained with Maple 2022.
 
\subsection{The second derivative  with Dirichlet boundary conditions}\label{sec:toy-model}
To begin with, we consider a toy model. Let $-A$ be the second derivative with Dirichlet boundary conditions on a bounded interval
\[
A=-\frac{\ud^2}{\ud x^2},\qquad D(A)=H^2(0,1)\cap H^1_0(0,1),
\]
which generates a positive, irreducible, compact (and analytic) $C_0$-semigroup. Therefore, 
\autoref{cor:landscape-abstr-domin}
can be applied: we will compare our estimates obtained with those obtained in~\cite{FilMay12,Ste17}.

Because $A$ is self-adjoint, $\|(\mu+A)^{-1}\|= (\dist(\mu,\sigma(-A)))^{-1}$ for all $\mu\in\rho(A)$ and, therefore,
\[
\inf_{\mu>s(-A)}|\lambda+\mu| \|(\mu+A)^{-1}\|=\lim_{\mu\to\infty}|\lambda+\mu| \|(\mu+A)^{-1}\|=1.
\]
One can expect that optimizing \eqref{eq:basic-fm-superabstr-lattice} over $\mu$ by allowing for large value of $\mu$ should be rewarding.
In practice, we often observe a remarkable phenomenon for $|\lambda+\mu| (\mu+A)^{-1}\mathbf 1$: the \emph{amplitude} induced for large $\mu$ tends to be more accurate, while the \emph{profile} of the ground state is better described for small $\mu$: the latter property was interpreted as a key to the study of localization features in~\cite{FilMay12}.

Heuristically,  \eqref{eq:landscape-bootstr} seems to offer a good trade-off.
For the above choice of $A$ we have $s(-A)=-\pi^2$, and applications of 
\autoref{cor:landscape-abstr-domin}
are presented in \autoref{fig:steiner?}. Let us stress the similarity of the optimal landscape function in \autoref{fig:steiner?} with that found, by different methods, in~\cite[Section~3.1]{HosQuaSte22}.

\begin{figure}[ht]
\begin{minipage}{6.5cm}
\includegraphics[width=6cm, height= 6cm]{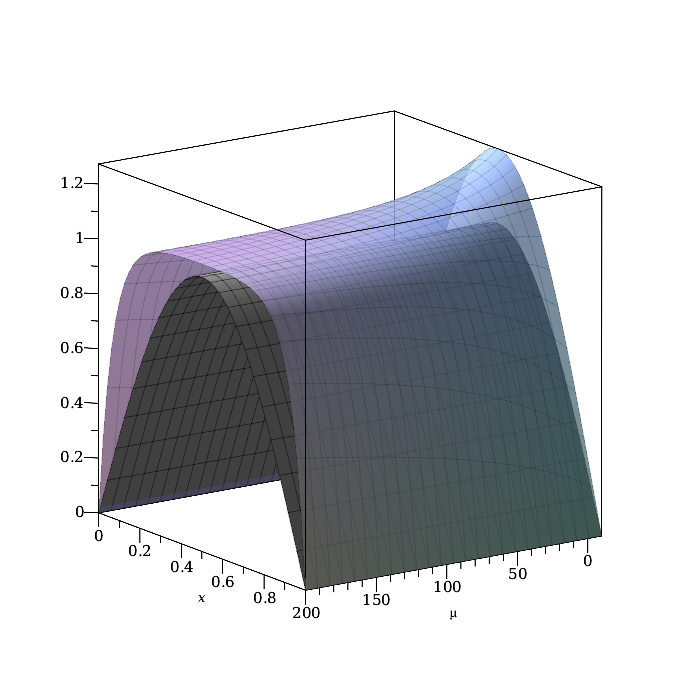}
\end{minipage}
\begin{minipage}{1pt}

\end{minipage}
\begin{minipage}{6.5cm}
\includegraphics[width=5cm, height= 4cm]{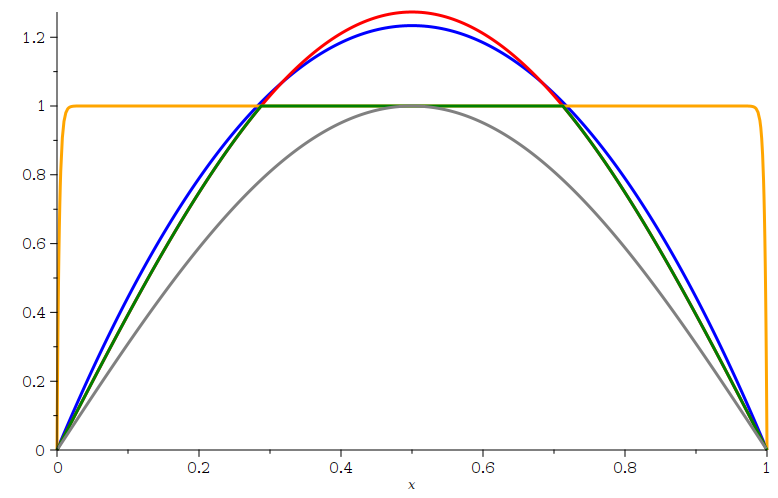}
\end{minipage}
\vspace{-15pt}
\caption{Left: The values of $(-s(-A)+\mu) (\mu+A)^{-1}\mathbf 1$ for $\mu\in [s(-A),200]$, where $-A$ is the second derivative with Dirichlet boundary conditions.\\
Right: An improved landscape function obtained by \autoref{cor:landscape-abstr-domin}: we have plotted $10^{-5}(-\pi^2+10^{-5}+A)^{-1}\mathbf 1$ and $10^{5}(-\pi^2+10^{5}+A)^{-1}\mathbf 1$ in red and orange, respectively, and and their pointwise minimum -- which by \autoref{cor:landscape-abstr-domin} dominates the ground state -- in green. The canonical landscape function $\pi^2 A^{-1}\mathbf 1$ from~\cite{FilMay12} is plotted as well (in blue).\\
For comparison, we have also plotted the actual ground state, i.e., $\sin(\pi \cdot)$ (in gray).}
\label{fig:steiner?}
\end{figure}

While it seems that, for the purpose of applying~\autoref{cor:landscape-abstr-domin}, there is no use in sampling $|-s(-A)+\mu| (\mu+A)^{-1}\mathbf 1$ at other positive values than $\mu=s(-A)+\sigma$ for $\sigma\approx 0+$ and $\sigma\approx +\infty$, further improvement can be achieved by using the anti-maximum principle mentioned in~\autoref{prop:landscape-abstr}.(\ref{item:amti}). Also, replacing $\mathbf 1$ by a  smoother $\rho$ may allow for yet more precise estimates: taking $\rho:= A^{-1}\mathbf 1$, which is inspired by \autoref{rem:irred}.(3),  seems to be a smart choice, see~\autoref{fig:coool!} (left). 

\begin{rem}
Observe that $A^{-1}\mathbf 1$ is a polynomial of degree two. When applying \autoref{prop:landscape-abstr}, different polynomials can, of course, be considered as candidates for $\rho$, as long as $D(A)\subset E_\rho$: this condition, however, is not always satisfied, as the choice of $\rho(x):=x^2(1-x)^2$ shows, as in this case $\varphi:=\sin(\pi \cdot)\leq c\rho$ fails to hold for any $c\ge 0$.
\end{rem}

In the case of the present $A$, it is even feasible to derive estimates based on the parabolic landscape function, as in \autoref{cor:landscape-abstr-domin}. Indeed, expanding the $C_0$-semigroup in Fourier series shows that \eqref{eq:landscape-bootstr} reads in this case
\begin{equation}\label{eq:bound-heat}
\frac{|\varphi(x)|}{\|\varphi\|_\infty}\le \frac{4}{\pi} \inf_{t>0}\sum_{k\ \mathrm{ odd }}\frac{\e^{t(\lambda-\pi^2 k^2)}}{ k}\sin(\pi k x)\qquad\hbox{for all }x\in [0,1],
\end{equation}
see~\autoref{fig:coool!} (right): this bound has been found already in~\cite[Section~3.1]{HosQuaSte22}. 

\begin{figure}[ht]
\begin{minipage}{5cm}
\includegraphics[width=5cm, height=4cm]{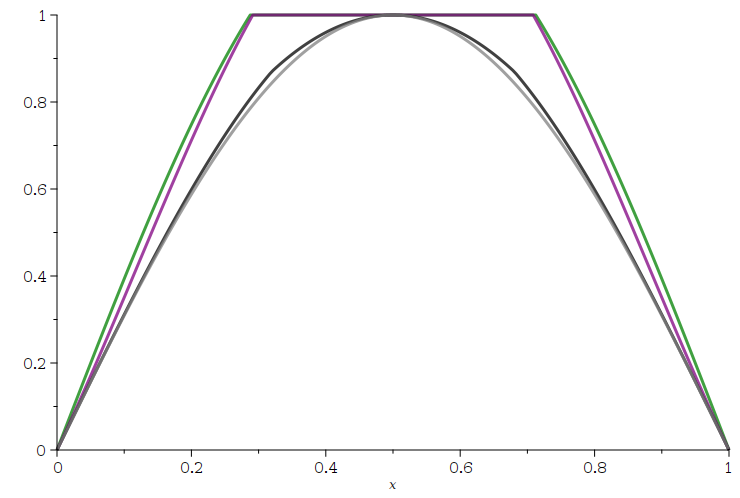}
\end{minipage}
\begin{minipage}{1cm}
\
\end{minipage}
\begin{minipage}{5cm}
\includegraphics[width=5cm, height=4cm]{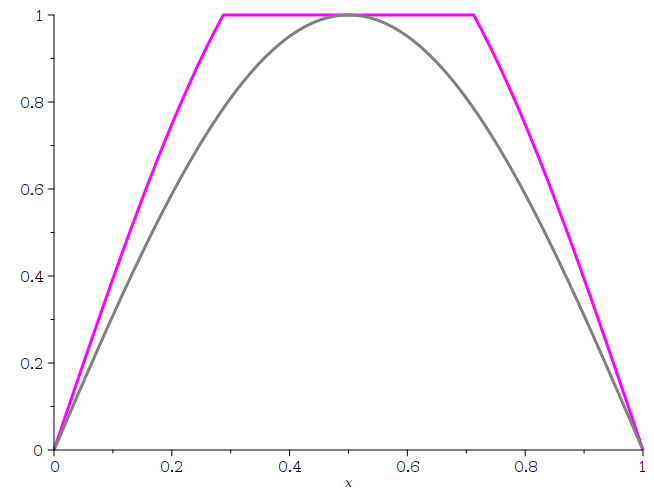}
\end{minipage}
\vspace{-15pt}
\caption{Left: In green: the best landscape function from \autoref{fig:steiner?}. In purple: a further improved landscape function obtained applying~\eqref{eq:landscape-bootstr}: we have taken the pointwise minimum of $-7(-\pi^2-7+A)^{-1}\mathbf 1$ and the optimal landscape function from~\autoref{fig:steiner?}.
We also present a landscape function obtained applying \eqref{eq:landscape-bootstr} with $\rho(x)=4x(1-x)$: we have plotted in black the pointwise minimum of $-7(-\pi^2-7+A)^{-1}\rho$ and $10^{5}(-\pi^2+10^{5}+A)^{-1}\rho$.\\
Right: A landscape function obtained applying the right estimate in \eqref{eq:landscape-bootstr}: we have plotted in magenta the pointwise minimum of $\e^{10^{-5}\pi^2}\e^{-10^{-5}A}\mathbf{1}$ and $\e^{10^{5}\pi^2}\e^{-10^{5}A}\mathbf{1}\simeq P_*\mathbf{1}$. The function $\e^{-tA}\mathbf{1}$ has been approximated truncating the series in~\eqref{eq:bound-heat} after the first 150 terms.
\\
In gray: the actual ground state (both left and right).\\
}\label{fig:coool!}
\end{figure}
\autoref{fig:higher} shows that our approach delivers reasonable pointwise eigenvector bounds even for higher eigenvalues. Indeed, by \autoref{cor:heatkernel} we can use these estimates to deliver upper bounds on the heat kernel, see \autoref{fig:heatkern}. As the resolution of the eigenvector bound \eqref{eq:landscape-superabstr-pos-domin-resolv} becomes lower and lower for higher eigenvalues, the bound cannot capture the degeneracy of the heat kernel as $t\to 0$, but is reasonably accurate for larger $t$.

\begin{figure}[ht]
\begin{minipage}{5.5cm}
\includegraphics[width=5cm, height= 4cm]{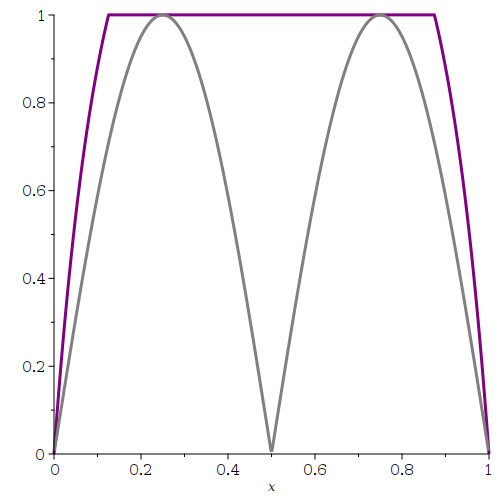}
\end{minipage}
\begin{minipage}{1cm}
\
\end{minipage}
\begin{minipage}{5.5cm}
\includegraphics[width=5cm, height= 4cm]{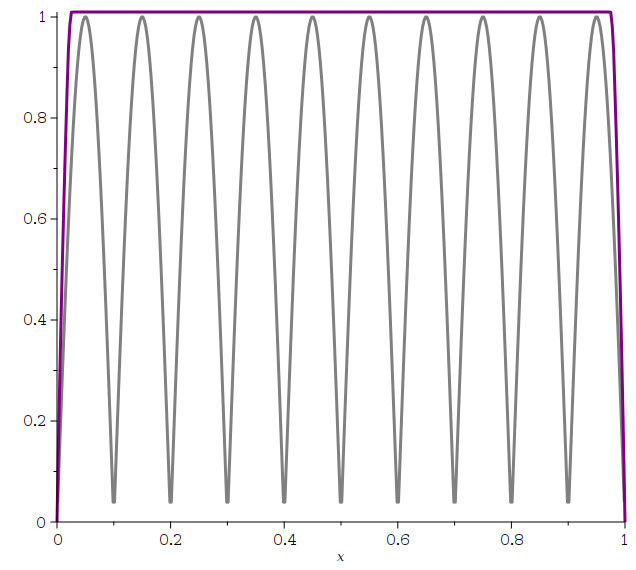}
\end{minipage}
\vspace{-15pt}
\caption{Improved landscape functions for absolute values of the second and tenth eigenvectors of $A$, again as in \eqref{eq:landscape-bootstr}. In purple:  The pointwise minimum of the values of $(\lambda_k+\mu) (\mu+A)^{-1}\mathbf 1$ for $\mu\in \{10^{-5},10,10^2,10^3,10^4,10^5\}$ for $k=2$ (left) and $k=10$ (right).
In gray: the absolute value of the actual eigenvectors.}
\label{fig:higher}
\end{figure}

\begin{figure}[ht]
\begin{minipage}{5.5cm}
\includegraphics[width=5.5cm, height= 5cm]{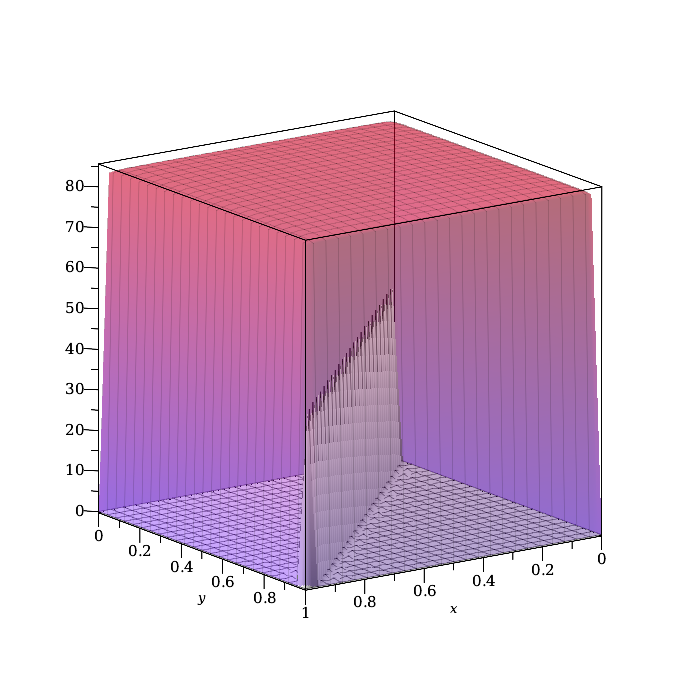}
\end{minipage}
\begin{minipage}{5.5cm}
\includegraphics[width=5.5cm, height= 5cm]{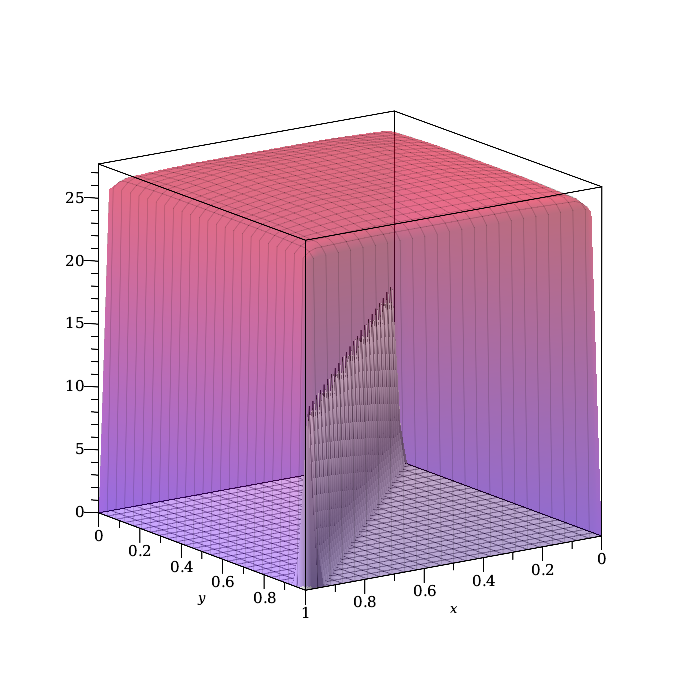}
\end{minipage}
\begin{minipage}{5.5cm}
\includegraphics[width=5.5cm, height= 5cm]{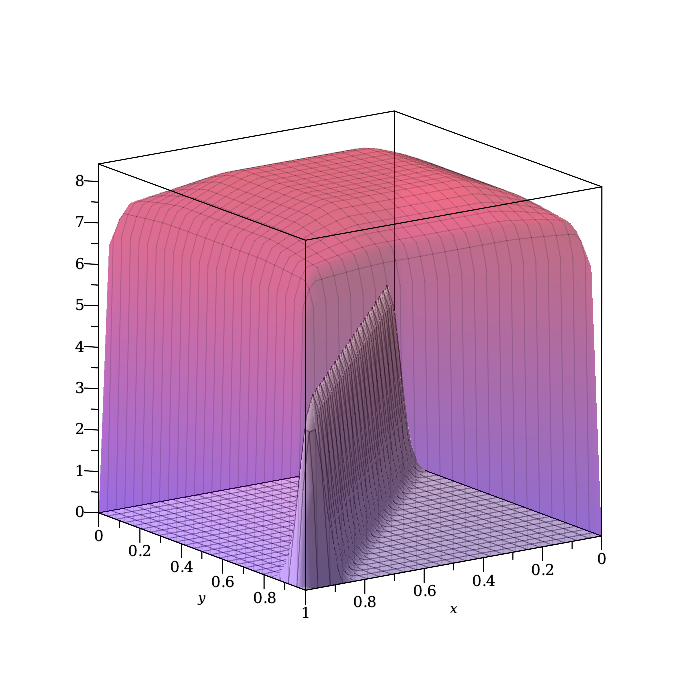}
\end{minipage}

\vspace{-12pt}
\begin{minipage}{5.5cm}
\includegraphics[width=5.5cm, height= 5cm]{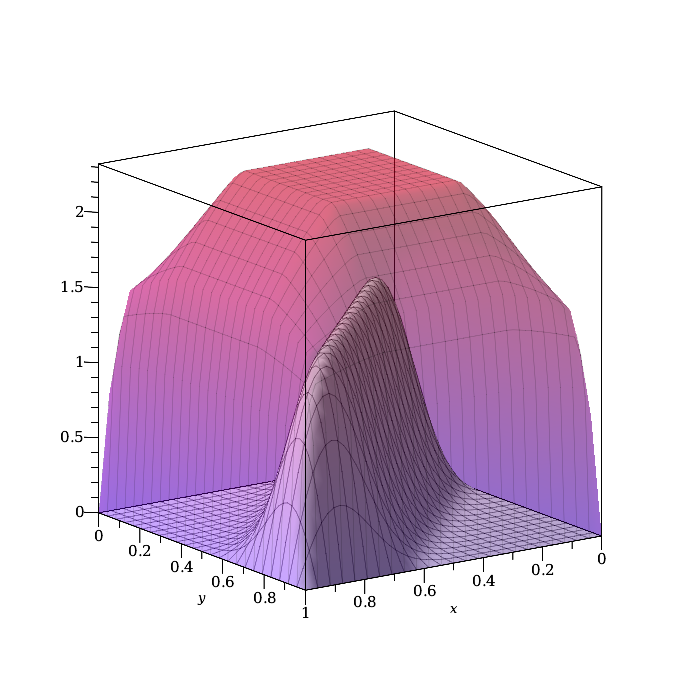}
\end{minipage}
\begin{minipage}{5.5cm}
\includegraphics[width=5.5cm, height= 5cm]{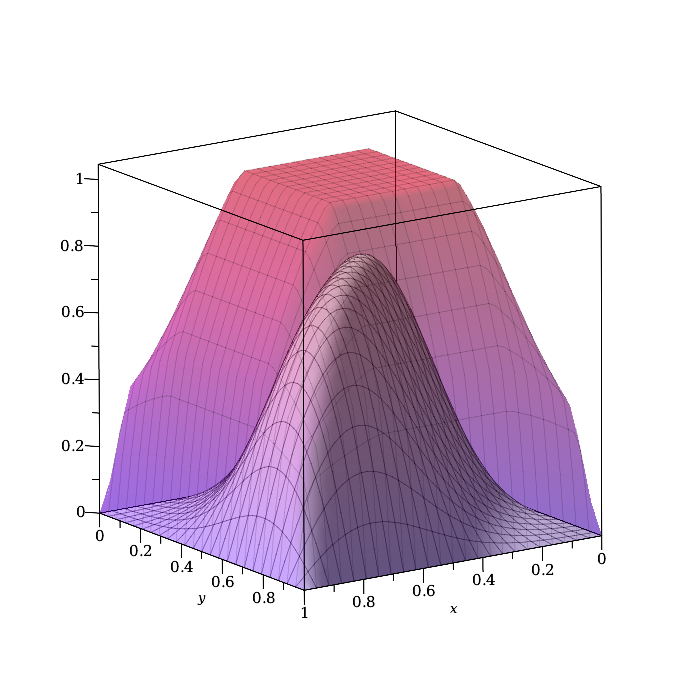}
\end{minipage}
\begin{minipage}{5.5cm}
\includegraphics[width=5.5cm, height= 5cm]{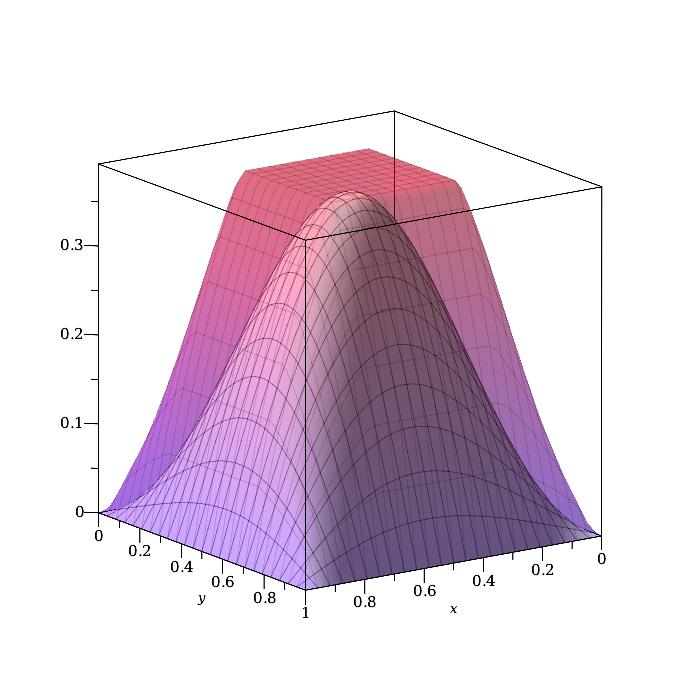}
\end{minipage}
\vspace{-15pt}
\caption{The actual heat kernel (in gray) for the second derivative with Dirichlet conditions on $(0,1)$ and its pointwise upper bound (in purple) given by \eqref{eq:filmay-heatkernel-invertible}, at time $t=10^{-5}$, $t=10^{-4}$, $t=10^{-3}$, $t=10^{-2}$, $t=3\cdot 10^{-1}$, $t=10^{-1}$; the plots of both functions are obtained truncating the corresponding series after 150 terms. 
}
\label{fig:heatkern}
\end{figure}

Finally, given any $C_0$-semigroup that is dominated by the heat semigroup $(\e^{-tA})_{t\ge 0}$ with Dirichlet conditions on $L^2(0,1)$, then its generator's principal value $\lambda_*$ can be estimated by \eqref{eq:giosmi-approx-general-parab}: we find 
\begin{equation}\label{eq:giosmi-approx-general-parab-appl}
\begin{split}
 \lambda_* &\ge 
- \inf_{t>0} \frac{\log\|\e^{-tA}\mathbf{1}\|_\infty}{t}
=
- \inf_{t>0}\frac{1}{t}\log  \frac{4}{\pi }\sum\limits_{k\ \mathrm{ odd }}^\infty \frac{\e^{-t\pi^2k^2}}{ k}.
\end{split}
\end{equation}
This estimate is possibly similar in spirit to that in~\cite[Section~4(a)]{LuSte17}, but the latter appears to be more accurate.

\subsection{$p$-Laplacians}\label{exa:main-nonlinear}
Let us apply our theory to a common nonlinear operator: the $p$-Laplacian $\Delta_p:=\nabla\cdot(|\nabla|^{p-2}\nabla)$ on a bounded open set $\Omega\subset \R^d$, for $p\in (1,\infty)$. We consider the energy functional $\mathcal E:L^2(\Omega)\to \R\cup \{\infty\}$ defined by
\[
\mathcal E:f\mapsto
\begin{cases}
\|\nabla f\|^p_{p},\quad &\hbox{if }f\in W^{1,p}_0(\Omega),\\
\infty,&\hbox{otherwise},
\end{cases}
\]
or else by
\[
\mathcal E:f\mapsto
\begin{cases}
\|\nabla f\|^p_{p}+\beta\|f_{|\partial\Omega}\|^p_p,\quad &\hbox{if }f\in W^{1,p}(\Omega),\\
\infty,&\hbox{otherwise},
\end{cases}
\]
for $\beta\in(0,\infty)$. Because $\mathcal E$ is $p$-homogeneous, its subdifferential is $(p-1)$-homogeneous: indeed, $\partial \mathcal E$ is (minus) the $p$-Laplacian with Dirichlet or Robin boundary conditions,
i.e.,
\begin{equation}\label{eq:bc-p}
f(z)=0 \qquad\hbox{or}\qquad \vert \nabla f \vert^{p-2} \frac{\partial f}{\partial \nu}(z)+\beta |f(z)|^{p-2}f(z)=0,\qquad z\in \Omega,
\end{equation}
 respectively: let us take $A:=\partial \mathcal E$. Then $A$ is accretive; it is invertible whenever endowed with Dirichlet or (for $\beta> 0$) Robin conditions.

\begin{prop}
Let $\Omega\subset \R^d$ be a bounded open domain.
Then, for each $p\in (1,\infty)$, each eigenpair $(\lambda,\varphi)$ of the $p$-Laplacian with either Dirichlet or Robin boundary conditions as in~\eqref{eq:bc-p}, for $\beta> 0$,   satisfies
\begin{equation}\label{eq:nonli-fm-p}
\frac{|\varphi(x)|}{\left\|\varphi\right\|_\infty}\le |\lambda|^\frac{1}{p-1}  (-\Delta_p)^{-1}\mathbf{1}(x)\qquad \hbox{for a.e.\ }x\in \Omega.
\end{equation}
Furthermore, the lowest eigenvalue $\lambda_{\min,p}$ satisfies
\begin{equation}\label{eq:giosmi-p}
\lambda_{\min,p}\ge  \|(-\Delta_p)^{-1}\mathbf{1}\|_\infty^{1-p}
\end{equation}
and, in the case of Dirichlet boundary conditions, we find for the Cheeger constant $h(\Omega)$
\begin{equation}\label{eq:giosmi-p-cheeg}
h(\Omega)\ge \lim_{p\to 1+} \|(-\Delta_p)^{-1}\mathbf{1}\|_\infty^{1-p}.
\end{equation}
\end{prop}

\begin{proof}
We only have to show that $A^{-1}=-\Delta_p^{-1}$ is order preserving: that is, we assume $\Delta_p u\le \Delta_p v$ and have to show that $u\ge v$. The proof is similar to that of~\cite[Theorem~2.1]{BobTak14}: integrate $\Delta_p v-\Delta_p u$ against $(v-u)^+$ and, using the Gauss--Green formulae and the boundary conditions, obtain
\[
\begin{split}
0 \leq \int_\Omega (\Delta_p v - \Delta_p u)(v-u)^+\ud x& = -\int_{\{v \geq u \}} (\vert \nabla v \vert^{p-2} \nabla v - \vert \nabla u \vert^{p-2} \nabla u) (\nabla v - \nabla u) \ud x\\
&\qquad -\beta \int_{\{v \geq u \}\cap \partial\Omega} (\vert  v \vert^{p-2}  v - \vert  u \vert^{p-2} u) ( v - u) \ud \sigma
\leq 0
\end{split}
\]
where the last inequality holds because $-\Delta_p$ is monotone: hence
$\int_\Omega (\Delta_p v - \Delta_p u)(v-u)^+  \ud x= 0$ and therefore $u \ge v$. 
The claim follows observing that all eigenvectors of $-\Delta_p$ with Dirichlet or Robin boundary conditions are essentially bounded, see~\cite[Theorem~4.1 and Corollary~4.2]{Le06}, and applying \autoref{prop:landscape-abstr-nonlin}.

 Finally, \eqref{eq:giosmi-p} holds by~\autoref{cor:spectral-estim}, whereas \eqref{eq:giosmi-p-cheeg} follows from \eqref{eq:giosmi-p} and~\cite[Corollary~6]{KawFri03}.
\end{proof}

In the special case of Dirichlet conditions, \eqref{eq:giosmi-p} has already been obtained in~\cite[Theorem~1]{BanCar94} in the case of $p=2$ and $d=2$; and in~\cite[Lemma~4.1]{GioSmi10} for general $p$, but for domains with smooth boundary only. Also the lower estimate on $\lambda_{\min,p}$ in~\cite[Theorem~4.2]{GioSmi10} remains true under our milder assumptions on $\partial\Omega$, as its proof is only based on Schwarz symmetrisation and the estimate~\eqref{eq:giosmi-p}.

Let us now focus on the 1D case with Dirichlet boundary conditions: if $\Omega=(0,1)$, then
\[
\left[(-\Delta_p)^{-1}\mathbf{1}\right](x)=-\frac{p-1}{p}\left(\left|x-\frac{1}{2}\right|^\frac{p}{p-1}-\left(\frac{1}{2}\right)^\frac{p}{p-1}\right),\qquad x\in (0,1),
\]
as shown by a direct computation. 
Its maximum is clearly attained at $x=\frac{1}{2}$, and \eqref{eq:giosmi-p} reads
\begin{equation}\label{eq:comp-ground-p}
\lambda_{\min,p} \ge 2^p\left(\frac{p}{p-1}\right)^{p-1}.
\end{equation}
Indeed, it is well-known that $\lambda_{\min,p}=(p-1)\pi_p^{\ p}$, where $\pi_p:=\frac{2\pi}{p \sin(\frac{\pi}{p})}$: a comparison of the left and right hand sides of~\eqref{eq:comp-ground-p} is shown in~\autoref{fig:sin4-lands} (left). Also, it is easy to see that the Cheeger constant of $\Omega=(0,1)$ is $h(0,1)=2$, while~\eqref{eq:giosmi-p-cheeg} yields 
\[
h(0,1)\ge \lim_{p\to 1+}2^p\left(\frac{p}{p-1}\right)^{p-1}=2.
\]
The estimate \eqref{eq:nonli-fm-p} is also remarkable because the $p$-trigonometric function $\sin_p$ -- which in turn yields the ($\|\cdot\|_\infty$-normalized) ground state 
\[
x\mapsto \sin_p(\pi_p x)
\]
 of the $p$-Laplacian with Dirichlet conditions on $(0,1)$ -- is only defined implicitly as the inverse of 
 \[
s\mapsto \int_0^s 	(1-t^p)^{-\frac{1}{p}}\ud t,
 \]
 see~\cite[Section~2.1]{LanEdm11}; 
even numerical plots of $\sin_p$ are not entirely trivial to obtain,  cf.~\cite{Fri03,GirKot14}.
The numerical values of $\lambda_{\min,p}^\frac{1}{p-1}\cdot (-\Delta_p)^{-1}\mathbf{1}$ obtained in~\autoref{fig:sin4-lands} (right) are in good agreement  with the plots of the ground state $\sin_p$ of $\Delta_p$ obtained in \cite[Abbildung~2]{Fri03} by Matlab 5.3.

\begin{figure}[ht]
\begin{minipage}{6.5cm}
\includegraphics[width=6cm, height= 5cm]{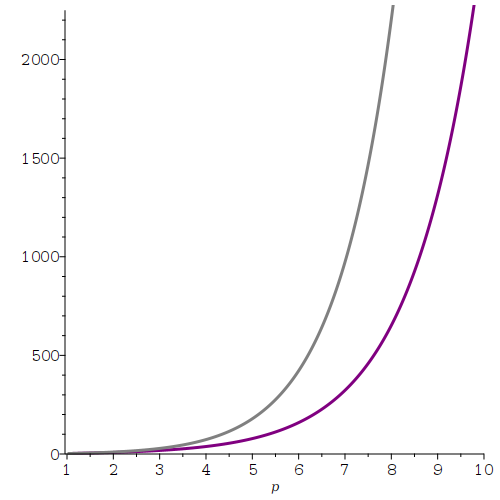}
\end{minipage}
\begin{minipage}{1pt}

\end{minipage}
\begin{minipage}{6.5cm}
\includegraphics[width=6cm, height=5cm]{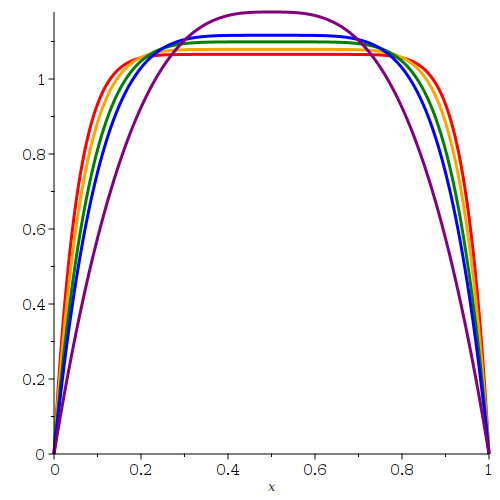}
\end{minipage}
\vspace{-15pt}
\caption{Left: The actual values of the minimal eigenvalue of the $p$-Laplacian with Dirichlet boundary conditions on $(0,1)$ (in gray) and its lower bound given by the $(1-p)$-th power of the maximum attained by the $p$-torsion function $(-\Delta_p)^{-1}\mathbf 1$ (in purple), plotted for $p\in [1,10]$.\\
Right: A plot of  $\lambda_{\min,p}^\frac{1}{p-1}\cdot (-\Delta_p)^{-1}\mathbf 1$ for 
$p=1.12$ (in red), $p=1.15$ (in orange), $p=1.2$ (in green), $p=1.25$ (in blue), and $p=1.5$ (in purple).}
\label{fig:sin4-lands}
\end{figure}

\medskip
Similar considerations hold, too, for the discrete $p$-Laplacian $\mathcal L_p$ with Dirichlet conditions on the boundary of any proper subgraph. Indeed, it is known that $\mathcal L_p$ is invertible and its inverse is order preserving, see e.g.~\cite[Theorem~A]{Par11}, so \autoref{prop:landscape-abstr-nonlin} can be applied. In the case of the $4$-Laplacian ($p$-Laplacian for $p=4$) on the unweighted path graph on 5 vertices $\mv_1,\mv_2,\mv_3,\mv_4,\mv_5$ with Dirichlet conditions at the extremal vertices $\mv_1,\mv_5$,  using symmetry arguments and running a numerical minimization on Mathematica Hua has found \cite{Hua22} that the ($\|\cdot\|\infty$-normalized) ground state $\varphi$ has the entries
\[
\varphi\approx \begin{pmatrix}
0\\
0.457504422\\
1\\
0.457504422\\
0
\end{pmatrix}
\]
with associated eigenvalue
\[
\lambda\approx 0.233665.
\]

By comparison, a numeric computation yields that the torsion function $\mathcal L_4^{-1}\mathbf 1$ for the same graph is the vector
\[
\mathcal L_4^{-1}\mathbf 1\approx \begin{pmatrix}
0\\
1.442249595\\
2.442249595\\
1.442249595\\
0
\end{pmatrix}
\]
and we derive from \autoref{prop:landscape-abstr-nonlin} the estimate
\[
\begin{pmatrix}
0\\
0.457504422\\
1\\
0.457504422\\
0
\end{pmatrix}\approx
\varphi\le \lambda^\frac{1}{3} \mathcal L_4^{-1}\mathbf 1\approx 0,615929807\cdot 
\begin{pmatrix}
0\\
1.144714243\\
1.938414769\\
1.144714243\\
0
\end{pmatrix}
\approx
 \begin{pmatrix}
0\\
0.705063623\\
1.193927435\\
0.705063623\\
0
\end{pmatrix}.
\]

\subsection{Magnetic Schrödinger operators}\label{sec:magn-schr-stein}
Let us finally provide a local bound for the eigenvectors of Schrödinger operators with both electric and magnetic potential $V$ and $a$, respectively: they are formally given by 
\[
H_{a,V}f:=(i\nabla +a)^2f+Vf.
\]
 Precisely introducing these operators is technical, but nowadays standard: based on quadratic form methods, on bounded open sets $\Omega\subset \R^d$ the realizations $H^{\mathrm{D}}_{a,V}$ (resp., $H^{\mathrm{N}}_{a,V}$) of  $H_{a,V}$ with Dirichlet (resp., Neumann) conditions can be defined as self-adjoint operators on $L^2(\Omega)$ whenever $a\in L^2_{\rm{loc}}(\Omega)$, $V=V_+-V_-$ with $V_+\in L^1_{\rm{loc}}(\Omega)$ and $V_-$ relatively form bounded with respect to the free Laplacian with Dirichlet (resp., Neumann) conditions, with relative bound $<1$.
Details can be found in~\cite{HunSim04} and references therein.

In 1D, finding a landscape function for such operators is trivial, since the magnetic potential can be gauged away by the local unitary transformation $U:f\mapsto \e^{i\int_0^\cdot a(s)\ud s}f$.
 Accordingly, the eigenvectors of the magnetic Schrödinger operators have, at any point $x$, same modulus as their non-magnetic counterparts and \autoref{prop:landscape-abstr} yields no novel information in this case. 
 
 In higher dimension, though, the issue of bounding eigenvectors of magnetic Schrödinger operators is subtler: 
 the torsion function $(H_{0,V})^{-1}\mathbf{1}$ of the non-magnetic Schrödinger operator $H_{0,V}$ has been proved in~\cite[Theorem~1]{HosQuaSte22} to be a landscape function for \textit{all} Schrödinger operators $H_{a,V}$ with same electric potential but arbitrary magnetic potential, and a corresponding parabolic landscape function has been obtained in~\cite[Theorem~2]{HosQuaSte22}.
 The proof is based on stochastic arguments, and in particular on a Feynman--Kac-type formula: unsurprisingly, all results in \cite{HosQuaSte22} hold under smoothness assumptions on $a,V$ and the boundary of $\partial \Omega$; only Dirichlet conditions and $V_-=0$ are allowed. Let us remove these restrictions and state the following general result.
 
\begin{prop}\label{prop:magn-land-appl-}
Let $\Omega\subset \R^d$ be a bounded open domain. Let the potentials $a,V$ be real-valued and such that $a\in L^2_{\rm{loc}}(\Omega)$, $V=V_+-V_-$ with $V_+\in L^1_{\rm{loc}}(\Omega)$, and let $V_-$ relatively form bounded, with relative bound $<1$, with respect to $H^{\mathrm{D}}_{0,0}$ (resp., $H^{\mathrm{N}}_{0,0}$). 
Then each eigenpair $(\lambda,\varphi)$ of $H^{\mathrm{D}}_{a,V}$ (resp., of $-H^{\mathrm{N}}_{a,V}$) satisfies
\begin{equation}\label{eq:magnetic-nonmagnetic-estim-Rd-d}
\frac{|\varphi(x)|}{\|\varphi\|_\infty}\le\inf_{\mu\ge 0} |\lambda+\mu| (\mu+H^{\mathrm{D}}_{0,V})^{-1}\mathbf 1(x)\qquad \hbox{ for a.e.\ }x\in \Omega
\end{equation}
\begin{equation}\label{eq:magnetic-nonmagnetic-estim-Rd-n}
\hbox{(resp.,}\quad \frac{|\varphi(x)|}{\|\varphi\|_\infty}\le \inf_{\mu>0}|\lambda+\mu| (\mu+H^{\mathrm{N}}_{0,V})^{-1}\mathbf 1(x)\qquad \hbox{ for a.e.\ }x\in \Omega.)
\end{equation}

Also, 
\begin{equation}\label{eq:magnetic-nonmagnetic-estim-Rd-d-parab}
\frac{|\varphi(x)|}{\|\varphi\|_\infty}\le \inf_{t>0}\e^{t \lambda}  \e^{-tH^{\mathrm{D}}_{0,V}} \mathbf{1}(x)\qquad \hbox{ for a.e.\ }x\in \Omega
\end{equation}
\begin{equation}\label{eq:magnetic-nonmagnetic-estim-Rd-n-parab}
\hbox{(resp.,}\quad \frac{|\varphi(x)|}{\|\varphi\|_\infty}\le \inf_{t>0}\e^{t \lambda}  \e^{-tH^{\mathrm{N}}_{0,V}} \mathbf{1}(x)\qquad \hbox{ for a.e.\ }x\in \Omega.)
\end{equation}
The above estimates remain true if the landscape functions in the right hand sides are replaced by $ (\mu + H^{\mathrm{D}}_{0,-V_-})^{-1}\mathbf{1}$ and $ (\e^{-t H^{\mathrm{D}}_{0,-V_-}})\mathbf{1}$, respectively
 (resp., $ (\mu+H^{\mathrm{N}}_{0,-V_-})^{-1}\mathbf{1}$ and $(\e^{-t H^{\mathrm{N}}_{0,-V_-}})\mathbf{1}$, respectively).

Furthermore, the bottom $\lambda_{*,a}$ of the spectrum of $H^{\mathrm{D}}_{a,V}$ satisfies
\begin{equation}\label{eq:comparison-bottom-spectrum-magnetic}
1\le \inf_{\mu\ge 0} |\lambda_{*,a,V}+\mu|  
 \|(\mu +H^{\mathrm{D}}_{0,V})^{-1}\mathbf 1\|_\infty.
\end{equation}
\end{prop}

\begin{proof}
To obtain the first four bounds, combine \cite[Theorem~1.1 and Remark~1.2.(iii)]{HunSim04} with~\autoref{cor:landscape-abstr-domin}.

The last estimate follows from \autoref{cor:pseudo-giosmi-linear}.
\end{proof}

In particular, $\lambda_{*,a}$  satisfies
\begin{equation}\label{eq:comparison-bottom-spectrum-magnetic-old}
|\lambda_{*,a,V}|\ge \frac{1}{\|(H^{\mathrm{D}}_{0,V})^{-1} \mathbf{1}\|_\infty},
\end{equation}
which already follows from \cite[Lemma~7.2.2]{Hel88} and \cite[Theorem~1]{BanCar94}.

\medskip
Likewise, we obtain a discrete version of \autoref{prop:magn-land-appl-}: we follow the terminology and notation  in \autoref{rem:modulusetc}.(3). In particular, given any graph and any magnetic signature $\alpha$ we denote by $\mathcal L_\alpha$ the corresponding magnetic Laplacian; and, especially, by $\mathcal L=\mathcal L_0$ and $\mathcal Q=\mathcal L_\pi$ the standard Laplacian and signless Laplacian, respectively. The realizations of $\mathcal L_\alpha$ (and $\mathcal L,\mathcal Q$) with Dirichlet conditions at the boundary of some proper subgraphs will be denoted by $\mathcal L^{\mathrm{D}}_\alpha$ (and, especially, $\mathcal L^{\mathrm{D}},\mathcal Q^{\mathrm{D}}$); it is well-known that they are invertible. All these operators are self-adjoint, and so are their versions with non-trivial, real-valued electric potential $\mathcal L_\alpha+V,\mathcal L+V,\mathcal Q+V$.

\begin{prop}\label{prop:magn-land-appl-discr}
Let $\mG=(\mV,\mE,\nu,\mu)$ be a finite, weighted, undirected graph, and let $\mG'$ some proper subgraph.
Then  for any real-valued potential $V=V_+ -V_-$ each eigenpair $(\lambda,\varphi)$ of $\mathcal L^{\mathrm{D}}_\alpha+V$ satisfies both
\begin{equation}\label{eq:magnetic-nonmagnetic-estim-discr-d}
\frac{|\varphi(\mv)|}{\|\varphi\|_\infty}\le \left[\inf_{\mu\ge 0} |\lambda+\mu| (\mu+\mathcal L^{\mathrm{D}}-V_-)^{-1}\mathbf 1\right](\mv)\qquad \hbox{ for all }\mv\in \mV
\end{equation}
and
\begin{equation}\label{eq:magnetic-nonmagnetic-estim-discr-d-signless}
\frac{|\phi(\mv)|}{\|\phi\|_\infty}\le \left[\inf_{\mu\ge 0} |\lambda_{\max}-\lambda+\mu| (\mu+\lambda_{\max}-\mathcal Q^{\mathrm{D}}-V_+)^{-1}\mathbf 1\right](\mv)\qquad \hbox{ for all }\mv\in \mV,
\end{equation}
where $\lambda_{\max}$ denotes the largest eigenvalue of $\mathcal L^{\mathrm{D}}_\alpha+V$.
\end{prop}

\begin{proof}
The assertions are an immediate consequence of the \autoref{cor:landscape-abstr-domin}, since
\[
\e^{-t(\mathcal L^{\mathrm{D}}_\alpha+V)}\le \e^{-t(\mathcal L^{\mathrm{D}}-V_-)}\quad\hbox{and}\quad \e^{t(\mathcal L^{\mathrm{D}}_\alpha+V)}\le \e^{t(\mathcal Q^{\mathrm{D}}+V_+)}\qquad\hbox{for all }t\ge 0,
\]
 in view of the domination properties recalled in \autoref{rem:modulusetc}.(3) and of the fact that $(\lambda_{\max}-\lambda,\varphi)$ is an eigenpair of $\lambda_{\max}-\mathcal{L}^{\mathrm{D}_\alpha}+V$.
\end{proof}

Even restricting to $\mu=0$, \eqref{eq:magnetic-nonmagnetic-estim-discr-d} sharpens the main estimate in~\cite{LemPacOvd20}, as $(\mathcal L^{\mathrm{D}}_{0})^{-1}$ is, by construction, dominated by the inverse of Ostrowski’s comparison matrix introduced in~\cite[Equation (3)]{LemPacOvd20}; while \eqref{eq:magnetic-nonmagnetic-estim-discr-d-signless} should be compared with the  high-energy eigenvectors bound by means of the dual landscape function obtained in \cite[Corollary~4]{LyrMayFil15}, which however only discusses  non-magnetic Laplacians on path graphs.

\bibliographystyle{plain}

\end{document}